\documentclass{article}

\usepackage{arxiv}

\usepackage[utf8]{inputenc} 
\usepackage[T1]{fontenc}    
\usepackage{hyperref}       
\usepackage{url}            
\usepackage{booktabs}       
\usepackage{amsfonts}       
\usepackage{nicefrac}       
\usepackage{microtype}      
\usepackage{lipsum}
\usepackage{graphicx}
\graphicspath{ {./images/} }
\usepackage[english]{babel}
\usepackage[utf8]{inputenc}

\usepackage{amsthm}
\usepackage{latexsym} \RequirePackage{amsthm}
\usepackage{amsmath}

\usepackage{amssymb}
\RequirePackage{makeidx}
\RequirePackage{graphicx}
\usepackage{txfonts}
\usepackage{kerkis,kmath}

\usepackage{enumitem}
\newcommand{\sm}{{\raise0.3ex\hbox{$\scriptstyle \setminus$}}}
\newcommand{\op}{\operatorname{op}}

\def\multiset#1#2{\ensuremath{\left(\kern-.2em\left(\genfrac{}{}{0pt}{}{#1}{#2}\right)\kern-.2em\right)}}

\numberwithin{equation}{section}

\newcommand{\R}{\mathbb{R}}
\newcommand{\C}{\mathbb{C}}

\DeclareMathOperator{\tr}{tr}

\newcommand{\rank}{\mathop{\operator@font rank}}

\newcommand{\e}{{\varepsilon}}
\newcommand{\vertiii}[1]{{\left\vert\kern-0.25ex\left\vert\kern-0.25ex\left\vert #1
    \right\vert\kern-0.25ex\right\vert\kern-0.25ex\right\vert}}
\makeatother

\newcommand{\B}[1]{\textbf{#1}}
\newcommand{\N}{\mathbb{N}} 
\newcommand{\E}{\mathbf{E}}
\newcommand{\PP}{\mathbf{P}} 

\newtheorem{thm}{Theorem}[section]
\newtheorem{lem}[thm]{Lemma}
\newtheorem{prop}[thm]{Proposition}
\newtheorem{cor}[thm]{Corollary}

\theoremstyle{definition}

\newtheorem{assumption}[thm]{Assumption}
\newtheorem{defn}[thm]{Definition}

\theoremstyle{remark}
\newtheorem{rem}[thm]{Remark}

\title{The limit of the operator norm for \\ random matrices with a variance profile\thanks{Research supported by the Hellenic Foundation for
Research and Innovation (H.F.R.I.) under the ``First Call for Research Projects to support Faculty members and Researchers and the procurement of high-cost research equipment grant, Project Number: 1034. \\ M.L. has also received funding from the European Union’s Horizon 2020 research and innovation program under the
Marie Sklodowska-Curie grant agreement No 101034255.}}

\author{Dimitris Cheliotis\thanks{National and Kapodistrian University of Athens, Department of Mathematics,
	Panepistimiopolis, Athens 15784, Greece,
  \texttt{dcheliotis@math.uoa.gr}} \\
   \And 
 \author [Michail Louvaris\thanks{Laboratoire d’informatique Gaspard Monge (LIGM / UMR 8049),
Université Gustave Eiffell,
 Marne-la-Vallée
  \texttt{michail.louvaris@univ-eiffel.fr}}}

\begin{document}
\maketitle

\begin{abstract}
 In this work we study symmetric random matrices with variance profile satisfying certain conditions.  We establish the convergence of the operator norm of these matrices to the largest element of the support of the limiting empirical spectral distribution. We prove that it is sufficient for the entries of the matrix to have finite only the $4$-th moment or the $4+\epsilon$ moment in order for the convergence to hold in probability or almost surely respectively. Our approach determines the behaviour of the operator norm for random symmetric or non-symmetric matrices whose variance profile is given by a step or a continuous function, random band matrices whose bandwidth is proportional to their dimension, random Gram matrices, triangular matrices and more.  

\end{abstract}


\section{Introduction}

The problem of understanding  the operator norm of a large random matrix with independent entries is multidisciplinary, occupying mathematicians, statisticians, physicists. On the mathematical side, tools from classical probability, geometric analysis, combinatorics, free probability and more have been used. The problem dates back to 1981, where in \cite{furedi1981eigenvalues} the convergence of the largest eigenvalue of renormalized Wigner matrices (symmetric, i.i.d. entries) to the edge of the limiting distribution was established when the entries of the matrix are bounded. Next, in \cite{bai1988necessary}, the authors gave necessary and sufficient conditions for the entries of a Wigner matrix to converge. The crucial condition was that the entries should have finite 4-th moment. Similar bounds have been given to non-symmetric matrices with i.i.d. entries. Then, the difference of the largest eigenvalue and its limit, after re-normalization, was proven to converge to the Tracy-Widow law in \cite{tracy1994level}. Later, universality results were established for sparse random matrix models, for example in \cite{huang2022edge} for random graphs and in \cite{sodin2010spectral} for random banded matrices. Moreover, sharp non-asymptotic results for a general class of matrices were established in \cite{bandeira2016sharp} and in \cite{bandeira2021matrix}.
\\  All the models mentioned above can be considered as random matrices with general variance profile, i.e., random matrices whose entries' variances can depend on the dimension of the matrix and the location of the element in the matrix. These models have also drawn a lot of attention lately, see for example \cite{cook2018non}, \cite{cook2022non}, where non-Hermitian models are considered. More specifically, assume that $A_N=(a_{i, j}^{(N)}), N\in\N^+,$ is a sequence of symmetric random matrices, with $a_{i, j}^{(N)}$ real valued having mean zero and variance $s_{i, j}^{(N)}$ bounded by a fixed number, say 1. Classically, the first question is whether the empirical spectral distribution of an appropriate normalization of $A_N$ (e.g., $A_N/\sqrt{N}$) converges to a nontrivial probability measure, as in Wigner's theorem. Nothing guarantees that, and one can construct examples where the sequence of the empirical spectral distributions does not converge. The work \cite{zhu2020graphon}, using the notion of graphons, gave conditions on the variance profile $s_{i, j}^{(N)}, i, j\in [N], N\in \N^+$ so that convergence takes place. 

The next, natural, question concerns the convergence of the largest eigenvalue to the largest element of the support of the limiting distribution. Again, this in not automatic but requires additional assumptions. It was established in the recent works \cite{AEKN2019}, \cite{EKS2019}, \cite{husson2022large}, \cite{DHG2024} (whose focus however is not this question) for some class of random matrices with a general variance profile under the assumption that the entries of the matrices have finite all moments (the first two works assume that each $a_{i, j}^{(N)}$ is sharp sub-Gaussian, the last two assume that for each $k\in\N^+$ there is a constant bounding the $2k$ moment of each $a_{i, j}^{(N)}$). In this paper, we generalize these results, i.e., we establish the convergence of the largest eigenvalue of general variance profile random matrices to the largest element of the support of the limiting empirical spectral distribution under general assumptions for the variance profile of the matrices. Regarding finiteness of moments, we assume only that $\sup_{N\in\N^+, i, j\in[N]} \E|a_{i, j}^{(N)}|^4<\infty$.  

\smallskip

\noindent \textbf{Notation}.
    For any $N\times N$ matrix $A=(a_{i, j})_{i, j\in [N]}\in \R^{N\times N}$ with eigenvalues $\{\lambda_{i}(A)\}_{i\in [N]}$, the measure  
    \[\mu_{A}:=\frac{1}{N}\sum_{i \in [N]}\delta_{\lambda_{i}(A)}\]
    will be called the Empirical Spectral Distribution (E.S.D.) of $A$. When the eigenvalues are real, write $\lambda_{\max}(A)$ for the maximum among them. We will use the following two norms on square matrices. For $A\in\R^{N\times N}$,
    \begin{align}
     |A|_{\op}&:=\max_{x\in \R^N: ||x||_2=1} ||Ax||_2=\sqrt{\lambda_{\max}(AA^T)}\\
    ||A||_{\max}&:=\max_{i, j\in[N]}|a_{i, j}|.
    \end{align}
It is easy to see that $|A|_{\op}\le N ||A||_{\max}$ and if the matrix $A$ is symmetric, then  
\begin{equation}
    |A|_{\op}:=\max_{i \in [N]}|\lambda_{i}(A)|. 
\end{equation}

\section{Statement of the results}

Throughout this section, $(A_N)_{N\in\N^+}$ is a sequence of symmetric random matrices with independent entries (up to symmetry), $A_N=(a_{i,j}^{(N)})_{i, j\in [N]}$ is an $N\times N$ matrix, and all $\{a_{i, j}^{(N)}: N\in \N^+, i, j\in [N]\}$ are defined on the same probability space and take real values. 

A standard assumption for the sequence is the following (see relation (2.2.1) in \cite{bai2010spectral}).

\begin{assumption} \phantom{} 
\label{BasicAssum}
\begin{enumerate}
    \item $\E a_{i, j}^{(N)}=0$ for all $N\in\N^+, i, j\in [N]$, and $\sup_{N\in \N^+, i, j\in [N]} \E|a_{i, j}^{(N)}|^2<\infty$.
    \item For any $\varepsilon>0$, 
    \begin{equation}
        \lim_{N\to\infty}\frac{1}{N^2}\sum_{i, j\in [N]} \E\Big\{|a_{i, j}^{(N)}|^2\mathbf{1}_{|a_{i, j}^{(N)}|\ge \varepsilon \sqrt{N}}\Big\}=0.
    \end{equation}
\end{enumerate}
\end{assumption}
This is satisfied in the case that $\{a_{i, j}^{(N)}: N\in \N^+, i, j \in[N], i\le j\}$ are i.i.d. with mean 0 and finite variance. But it is not enough to guarantee that the ESD of the appropriately normalized $A_N$ converges to a nontrivial limit. To state a sufficient condition for this, we introduce some notation that will be used throughout the work. 
\noindent We let 
\begin{equation}
s_{i,j}^{(N)}:=\E\{|a_{i,j}^{(N)}|^{2}\}
\end{equation} for all $N\in\N^+, i, j\in[N]$ and $V_0:=\sup_{N\in \N^+, i, j\in [N]} s_{i, j}^{(N)}\in[0, \infty)$.

Also, let $\mathbf{C}_k$ be the set of ordered rooted trees with $k$ edges (where $k\in\N$) of all non-isomorphic plane rooted trees with $k+1$ vertices, i.e. all trees with $k+1$ vertices, a vertex distinguished as a root and an ordering amongst the children of any vertex. The number of such trees is the $k-$th Catalan number, i.e.,
    \begin{equation}\label{CatalNum} |\mathbf{C}_k|=\frac{1}{k+1} \binom{2k}{k},\end{equation}
    and a trivial bound that we will use is $|\mathbf{C}_{k}|\le 2^{2k}$. For each such tree, we consider its vertices ordered $v_0<v_1<\cdots<v_k$ so that $v_0$ is the root, each parent is smaller than its children, and the children keep the order they have as vertices of an ordered tree. A labeling of such a tree is an ordered $k+1$-tuple $(\ell_0, \ell_1, \cdots, \ell_k)$ of different objects, the object $\ell_i$ is the label of vertex $v_i$.

A quantity of fundamental importance for the sequel is the following sum 
\begin{equation} 
M_N(k):=\sum_{T\in \mathbf{C}_k}\sum_{\substack{\mathbf{i}\in [N]^{k+1} \\ \text{ labeling of } T}} \prod_{\{i, j\}\in E(T)} s^{(N)}_{i,j}.\label{orismos kalon orwn}     
\end{equation}
$E(T)$ denotes the set of edges of the tree $T$. Note that $M_N(0)=N$ since by convention the product over an empty index set equals 1.

\begin{assumption}\label{ConvergenceOnTrees}
 There is a probability measure $\mu$ on $\R$ such that for each $k\in\N$ it holds \begin{equation}\label{assumptionchangedgraphon}
 \lim_{N\to\infty}\frac{M_N(k)}{N^{k+1}}=\int x^{2k}\, d\mu(x).
 \end{equation}
\end{assumption}
A tool for checking this assumption is explained in Remark \ref{RemConvOnTrees} below. 

If the sequence $(A_N)_{N\in\N^+}$ satisfies both Assumptions \ref{BasicAssum} and \ref{ConvergenceOnTrees}, then $\mu_{A_N/\sqrt{N}}\Rightarrow \mu$ with probability one (see the proof of  Theorem 3.2 of \cite{zhu2020graphon}). The measure $\mu$ is symmetric with compact support  contained in $[-2\sqrt{V_0}, 2\sqrt{V}_0]$. The compactness of the support follows from \eqref{assumptionchangedgraphon}, $M_N(k)\le |\mathbf{C}_k| N^{k+1} V_0^k$, and $|\mathbf{C}_k|\le 2^{2k}$.
Let
\begin{equation} \label{muInfDef}
    \mu_\infty:=\sup \text{suppt} \,\mu.
\end{equation}

We seek conditions under which the maximum eigenvalue of $A_N/\sqrt{N}$ converges to $\mu_\infty$ in probability. An easy argument will give us the lower bound, and since $\lambda_{\max}(A)\le |A|_{\op}$ for any symmetric matrix $A\in\R^{N\times N}$, it will be enough to prove the upper bound for the operator norm of $A_N/\sqrt{N}$.

For this purpose, we need stronger assumptions. The following is stronger than Assumption \ref{BasicAssum}.
\noindent \begin{assumption} \phantom{} \label{Genikes ypotheseis} 
\begin{itemize} 
    \item[(a)] $\E a_{i,j}^{(N)}=0$ for all $N\in\N^+, i, j \in[N], \sup_{N \in \N^+, i, j\in [N]}\E |a_{i,j}^{(N)}|^2\le 1$, and $\sup_{N \in \N^+, i, j\in [N]}\E |a_{i,j}^{(N)}|^{4}< \infty$.
    \item[(b)] For any $\e>0$ it is true that 
    \begin{equation}\label{MaxToZero}\lim_{N\to\infty} \sum_{i, j} \PP(|a_{i, j}^{(N)}|\ge \e \sqrt{N})=0.
     \end{equation}
\end{itemize}
\end{assumption}
\noindent Note that condition \eqref{MaxToZero} is satisfied if we assume that all $\{a_{i, j}^{(N)}: N\in\N^+, i, j\in[N]\}$ have the same distribution with finite 4-th moment.

%

We gain control over $|A_N|_{\op}$ through the traces of high moments of $A_N$, and the main difficulty, which the next conditions (Assumption \ref{Assumptions 1} and Assumption \ref{Assumption 2}) try to address, is how to connect these traces with $\mu_{\infty}$, which emerges out of $\{\lambda_i(A_N):i\in[N]\}$ only after we take $N\to\infty$.  

\begin{assumption}\label{Assumptions 1}
For every $N \in \N^+$ and $i, j \in [N]$ it is true that
\begin{equation} \label{VarInequalities}
s_{i,j}^{(N)}\leq \min\{s_{2i,2j}^{(2N)}, s_{2i-1,2j}^{(2N)},  s_{2i-1,2j-1}^{(2N)}\}.
\end{equation}
\end{assumption}
For example, this assumption is satisfied if $s_{i, j}^{(N)}=h(i/N, j/N)$ for all $N\in\N^+, i, j\in[N]$, where $h:[0, 1]^2\to[0, \infty)$ is a function decreasing separately in each variable.  

In order to give the next sufficient condition, we first give some definitions.
\begin{defn}\label{orismos graphon}
    (i)  We call graphon any Borel measurable function $W:[0,1]\times [0,1]\rightarrow \R$ which is symmetric and integrable. \\
(ii) For any bounded graphon $W$ and any multigraph $G=(V,E)$, we call isomorphism density from $G$ to $W$ the quantity
\begin{equation} \label{isomDensity}
t(G,W):=\int_{[0,1]^{|V|}}\prod_{\{i,j\} \in E}W(x_{i,}x_{j}) \prod_{i \in V}dx_i.
    \end{equation}    
\end{defn}     

Now, let $(A_{N})_{N\in\N^+}$ be a sequence of random matrices with elements having finite second moment. Each $A_N$ defines a graphon, $W_N$, through the relation
\begin{equation} \label{VarGraphon}
 W_{N}(x,y):=s^{(N)}_{\lceil Nx\rceil,\lceil Ny\rceil }
\end{equation}
for each $(x, y)\in[0, 1]\times[0, 1]$. For this relation, $\lceil 0 \rceil$ denotes 1.
\begin{assumption}\label{Assumption 2} There exists a graphon $W$ such that the $W_N$ of \eqref{VarGraphon} satisfies
\begin{equation} \label{GraphCountConv}
    \lim_{N \rightarrow \infty}t(T, W_{N})=t(T, W) 
\end{equation}
 for any finite tree $T$. Moreover, for any $D>0$ there exists some $C=C(D)\in (0, \infty)$ and $N_{0}=N_{0}(D)\in\N^+$ such that for any $N\geq N_{0}$ it holds
    \begin{equation}\label{fast convergence assumption}
    \int_{[0, 1]^2}|W_{N}(x, y)-W(x, y)|\, dx \, dy\leq C N^{-D}.
    \end{equation}
\end{assumption}
This assumption together with Assumption \ref{Genikes ypotheseis} implies Assumption \ref{ConvergenceOnTrees} (This will be explained in Lemma \ref{Protasi gia anisotita graphon}). Again, we denote by $\mu_\infty$ the maximum of the support of $\mu$.

The assumptions we made so far will lead to convergence in probability of the largest eigenvalue. Next we give some extra condition, which will lead to the almost sure convergence of the largest eigenvalue.
\begin{assumption}\label{assumfora.s.} $(A_{N})_{N\in\N^+}$ is a sequence of symmetric random matrices, the entries of each $A_N$ are independent (up to symmetry), and there exists a random variable $X$ with mean $0$, variance $1$, and finite $4+\delta$ moment for some $\delta>0$, which stochastically dominates the entries of $A_N$ in the following sense
    \begin{align}\label{stochdom}
        \PP(|\{A_{N}\}_{i,j}|\geq t) \leq \PP (|X|\geq t), \text{ for all } t\in [0,\infty), N\in \N^+, i,j\in [N].
    \end{align}
   
\end{assumption}

We are now ready to present our first main result.
\begin{thm}\label{to theorima}
    Let $(A_{N})_{N\in\N^+}$ be a sequence of matrices satisfying Assumption \ref{Genikes ypotheseis}. Then if either Assumptions \ref{ConvergenceOnTrees} and \ref{Assumptions 1} hold or Assumption \ref{Assumption 2} holds, it is true that
    \begin{align}\label{siglisi sto theorima}
    \lim_{N \rightarrow \infty} \frac{|A_{N}|_{\op}}{\sqrt{N}}=\mu_{\infty} \text{ in probability }\end{align}
where $\mu_\infty$ is defined in \eqref{muInfDef}.
    Moreover, if the sequence $(A_{N})_{N\in\N^+}$ satisfies Assumption \ref{assumfora.s.}, the convergence in \eqref{siglisi sto theorima} holds in the almost sure sense. 
\end{thm}
Note that Assumption \ref{Assumptions 1} is restrictive and does not cover several of the well-known and studied models. Thus, in what follows, we try to extend the domain of validity of  Theorem \ref{to theorima}. We first give two definitions.

For $N\in\mathbb{N}^+$ and $U\subset\mathbb[N]^2$:

$\bullet$ We call a $(x, y)\in U$ \textit{internal point} of $U$ if $\{(x+d_1, y+d_2): d_1, d_2\in\{-1, 0, 1\}\}\subset U$. We denote by $U^{o}$ the set of internal points of $U$. 

$\bullet$ We say that $U$ is \textit{axially convex} if $(i, j)\in U, (i, j') \in U, r\in [N], (r-j)(r-j')<0$ imply $(i, r)\in U$ and $(i, j)\in U, (i', j) \in U, r\in [N], (r-i)(r-i')<0$ imply $(r, j)\in U$.

\begin{defn}[Generalized step function variance profile]\label{defngenvar}
    Let $(A_N)_{N\in\N^+}$ be a sequence of symmetric random matrices, $A_N$ of dimension $N\times N$, with each element having zero mean and finite second moment. Moreover, suppose that there exists an $\N^+-$valued sequence $(d_N)_{N\in\N^+}$ with $\lim_{N \to \infty}d_{N}/N=0$ and such that for each $N$ there is a partition $\mathcal{P}_N:=\{\mathcal{B}_i^{(N)}:i=1, 2, \ldots, d_N\}$ of the grid $[N]^2$ consisting of $d_{N}$ axially convex sets with the following properties. 
    \begin{itemize}
        \item[(a)] If $A\in \mathcal{P}_N$ then $R(A):=\{(i, j): (j, i)\in A\}\in \mathcal{P}_N$.     
         \item[(b)] For any $m\in [d_{N}]$ there exists $f\in [d_{2N}]$ such that
        \begin{align}\label{growth of interior of convex sets}
            2\mathcal{B}^{(N)}_{m} \subset \mathcal{B}^{(2N)}_{f}. \end{align}          
        \item[(c)] For any $N\in \N$, $m\in [d_{N}]$ and $i\in [N]$ the line segment $x=i$ intersects $\mathcal{B}_{m}^{(N)}\setminus (\mathcal{B}_m^{(N)})^{\circ}$ at most $2$ times.   
    \end{itemize}
        Then if for all $(i,j)\in[N]^2$ the variance of the $(i,j)$-entry of $A_{N}$ is given by
        \begin{align}\label{defnofvariancesconvex}
            s_{i,j}^{(N)}:=\sum_{m\in [d_{N}]} s^{(N)}_{m}1_{(i, j)\in \mathcal{B}^{(N)}_{m}} 
            \end{align}
        for some set of numbers $\{s_{i}\}_{i\in [d_{N}]}$ so that $s_m^{(N)}=s_k^{(N)}$ if $R(\mathcal{B}^{(N)}_m)=\mathcal{B}^{(N)}_k$, we will call the sequence of matrices $(A_{N})_{N\ge1}$ random matrix model whose variance profile is given by a generalized step function. 
\end{defn}

The following Theorem is a corollary of Theorem \ref{to theorima} and gives results of the type \eqref{siglisi sto theorima} for the operator norm of the matrix

$\bullet$ $A_N$ when $A_N$ is a non-periodic band matrix with band size proportional to $N$ or has a step or continuous profile.

$\bullet$ $A_N A_N^T$ (i.e., Gram matrix) when $A_N$ is a rectangular matrix with step or continuous variance profile. 

Details are given after the next theorem and in subsection \ref{applications}.

\begin{thm}\label{Theorima gia sxedon kalous pinakes}
    Let $(A_N)_{N\in \N^+}$ be a random matrix model whose variance profile is given by a generalized step function. If it also satisfies Assumptions  \ref{ConvergenceOnTrees}, \ref{Genikes ypotheseis}, and for every $N\in\N$ and $(i,j)\in [N]^2$ it is true that 
    \begin{align}\label{extrassuconvex}
        s_{i,j}^{(N)}\leq s_{2i,2j}^{(2N)},
        \end{align}
then
 \begin{equation} \label{GSFInP} \lim_{N \rightarrow \infty}\frac{|A_{N}|_{\op}}{\sqrt{N}}=\mu_{\infty} \ \  \ \ \text{in probability},
 \end{equation}
where $\mu_\infty$ is defined in \eqref{muInfDef}. Moreover, if the sequence $(A_N)_{N\in\N^+}$ satisfies Assumption \ref{assumfora.s.} the convergence in \eqref{GSFInP} holds in the almost sure sense.
\end{thm}

For any $N\in \N^+$ and any two $N\times N$ matrices $A,B$ we will denote by $A\odot B$ their Hadamard product, which is the entry-wise product of $A,B$,i.e., the $N\times N$ matrix with entries
\begin{equation}
\{A\odot B\}_{i,j}=\{A\}_{i,j}\{B\}_{i,j} \text{ for all }i, j \in [N]. 
\end{equation} 
Note that Assumption \ref{assumfora.s.} is satisfied if $A_{N}$ can be written as
\begin{equation}
A_N=  \Sigma_{N}\odot A'_{N},  
\end{equation}
 where $A'_{N}$ is a sequence of symmetric 
 random matrices with i.i.d. entries all following the same law, with $0$ mean, unit variance and finite $4+\delta$ moment for some $\delta>0$ and for each $N$ the entries of $\Sigma_{N}$ are elements of $[0,1]$.
 
 Next, we study the operator norm of two widespread random matrix models. 

\begin{defn}[Step function variance profile] \label{defn rmt step function}

Consider
        
\begin{itemize}\item[a)] $m\in\N^+$ and numbers $\{\sigma_{p,q}\}_{p, q \in [m] }\in [0,1]^{m\times m}$ with $\sigma_{p, q}=\sigma_{q, p}$ for all $p, q\in[m]$. 
\item[b)]  For each $N\in\N^+$, a partition of $[N]$ into $m$ intervals   $\{I_p^{(N)}\}_{p \in [m]}$. The numbering of the intervals is such that $x<y$ whenever $x\in I_p^{(N)}, y\in I_q^{(N)}$ and $p<q$. Let $L_p^{(N)}$ and $R_p^{(N)}$ be the left and right endpoint respectively of $I_p^{(N)}$. 

\item[c)] Numbers $0=\alpha_0<\alpha_1<\cdots<\alpha_{m-1}<\alpha_m:=1$. We assume that $\lim_{N\to\infty}R_p^{(N)}/N=\alpha_p$ for each $p\in[m]$.

\item[d)] A random variable $X_0$ with $\E(X_0)=0, \E(X_0^2)=1$. 
\end{itemize}
For each $N\in\N^+$, define the matrix $\Sigma_N\in\R^{N\times N}$ by $(\Sigma_N)_{i, j}=\sigma_{p, q}$ if $i\in I_p^{(N)}, j\in I_q^{(N)}$, and let $\{A_{N}\}_{N \in \N^+}$ be the sequence of symmetric random matrices defined by
\begin{equation} A_{N}=\Sigma_{N}\odot A'_{N} \label{StepModel}
\end{equation}
where $A'_N$ is symmetric and its entries are independent (up to symmetry) random variables all with distribution the same as $X_0$. Then $(A_N)_{N \in \N^+}$ will be called \emph{symmetric random matrix model whose variance profile is given by a step function}. 
\end{defn}

Let $\hat I_p:=[\alpha_{p-1}, \alpha_p)$ for $p\in[m-1],$ and $\hat I_m:=[\alpha_{m-1}, 1]$. These intervals together with the numbers from a) determine a function $\sigma:[0,1]^{2}\to [0,1]$ as follows
\begin{equation}
    \sigma(x,y):=\sigma_{p, q} \text{ if } x\in \hat I_p, y\in \hat I_q.
\end{equation}
We call the function $\sigma^2$ the variance profile of the model.

\begin{defn}[Continuous function variance profile] \label{definition of rm cont}
           For
           \begin{itemize}
           \item[a)] a continuous and symmetric function $\sigma:[0, 1]^2\to[0, 1]$ (i.e, $\sigma(x, y)=\sigma(y, x)$ for all $x, y\in[0, 1]$),
           \item[b)] a sequence $(\Sigma_N)_{N\in\N^+}$ of symmetric matrices, $\Sigma_N\in[0, 1]^{N\times N}$, with the property
           \begin{equation} \label{ContinuousVPApprox}
               \lim_{N\to\infty}\sup_{1\le i, j\le N}\left|(\Sigma_N)_{i, j}-\sigma(i/N, j/N)\right|=0,
           \end{equation}
           \item[c)] a random variable $X_0$ with $\E(X_0)=0, \E(X_0^2)=1$,
           \end{itemize}
 consider the sequence $\{A_{N}\}_{N \in \N^+}$ of symmetric random matrices, $A_N$ of dimension $N\times N$, defined by
\begin{equation} A_{N}=\Sigma_{N}\odot A'_{N} \label{ContinuousModel}
\end{equation}
where the entries of $A'_N$ are independent (up to symmetry) random variables all with distribution the same as $X_0$. Then we say that $(A_N)_{N \in \N^+}$ is a \emph{random matrix model whose variance profile is given by a continuous function.} Again, we call the function $\sigma^2$ the variance profile. \end{defn}

\begin{rem}[Checking Assumption \ref{ConvergenceOnTrees}] \label{RemConvOnTrees}
A sufficient condition for the validity of Assumption \ref{ConvergenceOnTrees} is that $(A_N)_{N\in\N^+}$ satisfies Assumption \ref{BasicAssum} and there is a graphon $W$ such that $W_N\to W$ almost everywhere in $[0, 1]\times[0, 1]$.

Indeed, the bounded convergence theorem gives that $t(T, W_N)\to t(T, W)$ for all trees. Then Theorem 3.2 (a) of \cite{zhu2020graphon} shows that the ESD of $A_N/\sqrt{N}$ converges almost surely weakly to a probability measure $\mu^{\sqrt{W}}$ whose $2k$ moment equals
\begin{equation}
    \lim_{N\to\infty} \sum_{T\in \mathbf{C}_k}t(T, W_N)
\end{equation}
while the moments of odd order are 0. Then, for each $T\in \mathbf{C}_k$,
\begin{equation}
0\le t(T, W_N)-
\sum_{\substack{\mathbf{i}\in [N]^{K+1} \\ \text{ labeling of } T}} N^{-k-1} \prod_{\{i, j\}\in E(T)} s^{(N)}_{i,j}=O(1/N),
\end{equation}
because $ t(T, W_N)$ is simply the same as the sum in the previous relation with the only difference that $\textbf{i}$ is not required to be a labeling, i.e., it can have repetitions. It follows that Assumption \ref{ConvergenceOnTrees} holds. As we remarked after \eqref{assumptionchangedgraphon}, $\mu^{\sqrt{W}}$ is symmetric and has bounded support. Denote by $\mu^{\sqrt{W}}_{\infty}$ the largest element of the support.
\end{rem}

\quad If, in the two models above, $X_0$ has finite $4+\delta$ moment for some small $\delta>0$, then it is easy to see that the sequence $(A_{N})_{N\in\N^+}$ satisfies Assumptions \ref{Genikes ypotheseis}. It also satisfies Assumption \ref{ConvergenceOnTrees} because it satisfies Assumption \ref{BasicAssum} and, in both cases, $W_N(x, y)$ converges to $\sigma^2(x, y)$ for almost all $(x, y)\in[0, 1]\times[0, 1]$, thus the preceding remark applies.  

Our result for the model \eqref{StepModel} is the following.
 
\begin{thm}\label{theoremstepfunction}
Let $(A_N)_{N\in\N^+}$ be a random matrix model whose variance profile is given by a step function as above. Assume that $X_0$ has mean value 0, variance 1, and finite $4+\delta$ moment, for some small $\delta>0$. Then it is true that
\begin{equation} \label{StepProfLimit}
      \lim_{N \rightarrow \infty}\frac{|A_{N}|_{\op}}{\sqrt{N}}=\mu^{\sigma}_{\infty} \ \  \ \ \text{a.s.}
\end{equation}
\end{thm}

The previous theorem together with an approximation result that we prove in Section \ref{ContModelSection} (Proposition \ref{symperasma gia proseggiseis pinakon}) has the following consequence for the model \eqref{ContinuousModel}.
           \begin{cor}\label{theoremconfunction}
Let $(A_N)_{N\in\N^+}$ be a sequence of matrices whose variance profile is given by a continuous function. If $X_0$ has mean zero, variance one, and finite $4+\delta$ moment, then  
\begin{equation} \label{ContProfLimit}
    \lim_{N \rightarrow \infty}\frac{|A_{N}|_{\op}}{\sqrt{N}}=\mu^{\sigma}_{\infty} \ \  \ \ \text{a.s.}
\end{equation}
               \end{cor}

\begin{rem} 1) 
Theorem \ref{theoremstepfunction} covers the cases in the Wigner matrix model [i. e., $A_N:=(a_{i, j})_{i, j\in[N]}$ with $\{a_{i, j}: 1\le i\le j\le N, N\in\N^+\}$ i.i.d. with $\E(a_{1, 1})=0, \E(a_{1, 1}^2)=1$] where $\E(|a_{1, 1}|^{4+\delta})<\infty$ for some $\delta>0$. Recall that the necessary and sufficient condition for the validity of \eqref{StepProfLimit} in that model is $\E(|a_{1, 1}|^4)<\infty$.

2) Corrolary \ref{theoremconfunction} holds also in the case that the function $\sigma$ of Definition \ref{definition of rm cont} is piecewise continuous in a sense explained in the end of Section \ref{ContModelSection}.

\end{rem}

\section{Analysis of high order moments}\label{highordermoments}

Assume at the moment that the entries of $A_N$ have finite moments of all orders.

We will relate the largest eigenvalue with a high moment of the measure $\mu_N$ and at the same time this moment will be controlled by $\mu_\infty$. In general, for $k\in\N$, it is true that
\begin{equation}\label{geniko athrisma}
    \E \tr (A_N^{2k})= \sum_{i_{1},i_{2},........,i_{2k} \in [N]} \E \left(\prod_{l=1}^{2k}a_{i_{l},i_{l+1}}^{(N)}\right) \end{equation}
with the conventions that $i_{2k+1}=i_{1}$, when $k=0$ the sum is only over $i_1\in[N]$, and the product over an empty set equals 1. 


Now, for a term with indices $i_1, i_2, \ldots, i_{2k}$, we let $\B{i}:=(i_1, i_2, \ldots, i_{2k})$ and $X(\B{i}):=\prod_{l=1}^{2k}a_{i_{l},i_{l+1}}^{(N)}$. For such an $\B{i}$ we also use the term cycle. Then consider the graph $G(\B{i})$ with vertex set
$$V(\B{i})=\{i_1, i_2, \ldots, i_{2k}\}$$ and set of edges
\begin{equation}
\{\{i_r, i_{r+1}\}: r=1, 2, \ldots, 2k \}.
\end{equation}
%

\noindent As explained in \cite{bai2010spectral} (in the proof of relation (3.1.6) there,
pages 49, 50 or in Theorem 3.2 of \cite{zhu2020graphon}), the limit 
$$\lim_{N\to\infty} \frac{1}{N^
{k+1}}\E\tr(A^{2k}) $$
remains the same if in the sum of \eqref{geniko athrisma} we keep only the summands such that 
\begin{equation} \label{treeCondition}
\text{the graph $G(\B{i})$ is a tree with $k+1$ vertices} 
\end{equation}
Then, necessarily, the path $i_1\to i_2 \to \cdots \to i_{2k} \to i_1$ traverses each edge of the tree exactly twice, in opposite
directions of course. 
Such a $G(\B{i})$ becomes an ordered rooted tree if we mark $i_1$ as the root and order children of the same vertex according to the order they appear in the cycle. \\ 
Cycles $\B{i}$ that don't satisfy \eqref{treeCondition} we call bad cycles. So, for $k\in\N$, the sum in \eqref{geniko athrisma} can be written as
\begin{align}\label{xoristo athrisma}
\E \tr(A^{2k})&=M_N(k)+B_N(k),\\
\intertext{where}
M_N(k)&:=\sum_{T\in \mathbf{C}_k}\sum_{\mathbf{i}\in [N]^{(2k)\vee 1}: G(\mathbf{i})\sim T} \prod_{\{i, j\}\in E(G(\mathbf{i}))} s^{(N)}_{i,j}, \label{GoodCycSum} \\
B_N(k)&:=\sum_{\mathbf{i}\in[N]^{2k}: \text{bad cycle}}\E X(\B{i}). \label{oi kakoi oroi}
\end{align}
Recall that $\mathbf{C}_k$ are the ordered rooted trees with $k$ edges and $G(\mathbf{i})\sim T$ means that the graphs are isomorphic as ordered rooted trees. Note also that $M_N(k)$ has already been defined in \eqref{orismos kalon orwn} but the two definitions for it agree. Also, $M_N(0)=N, B_N(0)=0$.\\
The plan is to control the expectation of the trace in \eqref{xoristo athrisma} through an appropriate bound involving various $M_N(j)'s$. To control the term $B_N(k)$, we adopt the analysis of Section 2.3 of \cite{tao2012topics}.  
\begin{prop}\label{protasi gia megalo trace}
  Let $A_{N}$ be an $N\times N$ symmetric random matrix with independent entries (up to symmetry) and with $\E(a_{i, j}^{(N)})=0, s_{i, j}^{(N)}\le1$ for all $N\in\N, i, j\in [N]$. Assume additionally that the absolute value of the entries of the matrix are all supported in $[0,CN^{\frac{1}{2}-\epsilon}]$ for some $\epsilon>0$. Then for all $N$ large enough and all integers $1\le k<N$ it is true that
  \begin{equation}\label{BadCycBnd} |B_{N}(k)|\le \sum_{s=1}^{k} (4k^{5})^{2k-2s} \left(CN^{\frac{1}{2}-\epsilon}\right)^{2k-2s} \sum_{t=1}^{(s+1)\wedge k}(4k^4)^{4(s+1-t)}M_{N}(t-1).\end{equation}
\end{prop}
\begin{proof} We bound each term of the sum defining $B_N(k)$. Take a bad cycle $\B{i}$ and let
\begin{itemize}
\item t: the number of vertices of $G(\B{i})$,
\item $s$: the number of the edges of $G(\B{i})$,
\item $e_1, e_2, \ldots, e_s$: the edges of $G(\B{i})$ in order of appearance in the cycle,
\item $a_1, a_2, \ldots, a_s$:  the multiplicities of $e_1, e_2, \ldots, e_s$ in the cycle. 
\end{itemize}
That is, $a_q$ is the number of times the (undirected) edge $e_q$ appears in the cycle. Note that $t\le s+1$ (true for all graphs) and $t\le k$ because the cycle is bad.

Additionally, in case $t\ge2$, we let $T(\B{i})$ be the rooted ordered tree obtained from $G(\B{i})$ by keeping only edges that lead to a new vertex at the time of their appearance in the cycle. The root is $i_1$ and we declare a child of a vertex smaller than another if it appears earlier in the cycle. In case $t=1$, $T(\B{i})$ is the graph with one vertex, $i_1$, and one edge (loop) with end vertices $i_1, i_1$. Thus, $T(\B{i})$ has $t$ vertices and $1\vee(t-1)$ edges. 

To bound $|\E X(\B{i})|$, notice that if any of $a_1, a_2, \ldots, a_s$ is 1, we have $\E X(\B{i})=0$ by the independence of the elements of $A_N$ and the zero 
mean assumption. We assume therefore that all multiplicities are at least 2. Using the information about the mean, variance, and support of $|a_{i, j}^{(N)}|$, we get that for any integer $a\ge 2$ it holds 
$\E(|a_{i, j}^{(N)}|^a)\le 
(C_1 N ^{1/2-\epsilon})^{a-2} 
s_{i, j}^{(N)}.$ Thus 
\begin{equation}\E|X(\B{i})|= \prod_{q=1}^s \E |X_{e_q}|^{a_q} \le (C N^{1/2-\e})^{a_1+\cdots+a_s-2s}\prod_{\{i, j\}\in E(G(\B{i}))} s_{i, j}^{(N)}\le (C N^{1/2-\e})^{2k-2s}\prod_{\{i, j\}\in E(T(\B{i}))} s_{i, j}^{(N)}.
\end{equation}
In the second inequality, we used the fact that $s_{i, j}^{(N)}\in[0, 1]$ for all $i, j, N$.
For integers $s, t\ge1, a_1, \ldots, a_s\ge2$ and $T\in \mathbf{C}_{t-1}$ let 
\begin{equation} \label{SpecBadCycles}
N_{T, a_1, a_2, \ldots, a_s}=\begin{array}{l} \text{the number of bad cycles with $T(\B{i})\sim T$, indices $1, 2, \ldots, t$, appearing in this order,} \\ \text{and edge multiplicities } a_1, a_2, \ldots, a_s.
\end{array}
\end{equation}
Using the bound on $N_{T, a_1, a_2, \ldots, a_s}$ provided by Lemma \ref{CycleCountLemma}, we obtain
\begin{align} \label{FirstGraphBound}
|B_N(k)|
&\le \sum_{s=1}^k \sum_{t=1}^{k\wedge(s+1)}     \sum_{a_1, a_2, \ldots, a_s}  (C N^{1/2-\e})^{2k-2s} \left\{\mathbf{1}_{t=1}\sum_{i\in [N]} s_{i, i}^{(N)}+\mathbf{1}_{t\ge2} \sum_{T\in \B{C}_{t-1}} N_{T, a_1, a_2, \ldots, a_s}  \sum_{\B{i}\in [N]^{2k}:T(\B{i})\sim T} \prod_{\{i, j\}\in E(T(\B{i}))} s_{i, j}^{(N)}\right\}\\
&\le \sum_{s=1}^k \sum_{t=1}^{k\wedge(s+1)}     \sum_{a_1, a_2, \ldots, a_s}  (C N^{1/2-\e})^{2k-2s} \left\{\mathbf{1}_{t=1}\sum_{i\in [N]} s_{i, i}^{(N)}+\mathbf{1}_{t\ge2} \sum_{T\in \B{C}_{t-1}} N_{T, a_1, a_2, \ldots, a_s}  \sum_{\B{i}\in [N]^{2(t-1)}:T(\B{i})\sim T} \prod_{\{i, j\}\in E(T(\B{i}))} s_{i, j}^{(N)}\right\}\\
& \le \sum_{s=1}^k \sum_{t=1}^{k\wedge(s+1)}   \sum_{a_1, a_2, \ldots, a_s} (C N^{1/2-\e})^{2k-2s}   (4k^4)^{4(s+1-t)+2(k-s)} M_N(t-1). \label{FirstGraphBound3}
\end{align}
We used here the fact that $s_{i, i}^{(N)}\le1$, so that $\sum_{i\in [N]} s_{i, i}^{(N)}\le N=M_N(0)$.
The inside sum in \eqref{FirstGraphBound3} is over all $s$-tuples of integers $a_1, a_2, \ldots, a_s$ greater than or equal to 2 with sum $2k$. By subtracting 2 from each $a_i$, we get an $s$-tuple of non-negative integers with sum $2k-2s$.  The number of such $s$-tuples is $\multiset{s}{2k-2s}$(combinations with repetition), which is at most $s^{2(k-s)}\le k^{2(k-s)}$. Thus the above sum is bounded by
\begin{equation}
    \sum_{s=1}^k (4k^5)^{2(k-s)} (C N^{1/2-\e})^{2k-2s} \sum_{t=1}^{k\wedge(s+1)}    (4k^4)^{4(s+1-t)}  M_N(t-1).
\end{equation}
\end{proof}

\begin{prop} \label{PropSRCondition}  Let $(A_{N})_{N\in\N^+}$ be a sequence of symmetric matrices and $R>0$ so that the sequence satisfies Assumption \ref{Genikes ypotheseis} and the following condition $\Sigma(R)$: \\
For each $C_1>0$ there are $C_2>0$ and $N_0\in\N^+$ such that
\begin{equation} \label{SRCondition}
    M_N(k)\le C_2 N^{k+1} R^{2k} 
\end{equation}
for all $N, k\in\N^+$ with $N\ge N_0$ and $1\le k\le C_1 \log N$. \\
Then for each $\epsilon>0$, it holds
\begin{equation}    
\lim_{N\to\infty} \mathbb{P}\left(\frac{|A_{N}|_{\op}}{\sqrt{N}}\geq R (1+\epsilon)\right)=0.
\end{equation}
\end{prop}

\begin{proof}
Fix $\eta\in(0, 1/8)$ and define the $N\times N$ matrices $A_N^{\le}, A_N^{>}$ by
\begin{align}
(A_N^{\le})_{i, j}:=a^{(N)}_{i,j}1_{|a_{i,j}|\leq N^{\frac{1}{2}-\eta}},\\
(A_N^{>})_{i, j}:=a^{(N)}_{i,j}1_{|a_{i,j}|>N^{\frac{1}{2}-\eta}}
\end{align}
for all $i, j\in[N]$. For a random matrix $H:=(h_{i, j})$, $\E H$ denotes the matrix whose $(i, j)$ element is $\E h_{i, j}$ provided that the mean value of $h_{i, j}$ can be defined. Note that
    \begin{equation}\label{anisotita weyl gia mikra mesi timi kai megala}
\frac{1}{\sqrt{N}}|A_{N}|_{\op} \leq \frac{1}{\sqrt{N}} \left( |A_{N}^{\le}-\E A_{N}^{\le}|_{\op}+ |\E A_{N}^{\le}|_{\op}+ |A_{N}^{>}|_{\op}\right).
\end{equation}
We will bound the three terms in the right hand side of the last inequality. For the first two, we use only Assumption \ref{Genikes ypotheseis} and the arguments in the proof of Theorem 2.3.23 in \cite{tao2012topics}.

1) The term $N^{-\frac{1}{2}}|\E (A^{\leq }_{N})|_{\op}$ is a deterministic sequence that converges to 0 because, since $a_{i, j}^{(N)}$ is centered, we have
\begin{equation}\label{Deterministicto0SR}
       |(\E A^{\leq}_{N})_{i,j}|= |(\E A^{>}_{N})_{i.j}|\leq N^{-3 (\frac{1}{2}-\eta)} \sup_{N}\max_{i,j} \E |a_{i,j}^{(N)}|^{4}.
\end{equation}
And using the inequality $|C|_{\op}\le N ||C||_{\max}$, we get that 
$$|\E (A^{\leq }_{N})|_{\op} \le N^{3\eta-\frac{1}{2}} \sup_{N}\max_{i,j} \E |a_{i,j}^{(N)}|^{4}\overset{N\to\infty}{\to}0.$$

2) The term $N^{-\frac{1}{2}} |A^{>}_{N}|_{\op}$ converges to $0$ in probability. Indeed, for any $\delta_1>0$, 
\begin{align}
\PP(|A^{>}_{N}|_{\op}>\delta_1 \sqrt{N})\le &\PP(|a_{i,j}^{(N)}| > \delta_1 \sqrt{N} \text{ for some } i,j \in [N])\\+&\PP(|A^{>}_{N}|_{\op}>\delta_1 \sqrt{N}\text{ and } |a_{i,j}^{(N)}| \le \delta_1 \sqrt{N} \text{ for all } i,j \in [N]). \label{opnormbound2}
\end{align}
The first quantity goes to zero as $N\to\infty$ because of \eqref{MaxToZero}. For the second, it is an easy exercise to show that if each entry of a matrix $M$ has absolute value at most $a$ and each row and column of $M$ has at most one non-zero element then $|M|_{\op}\le a$ (use the expression $|M|_{\op}=\sup_{x:||x||_2=1} ||Mx||_2$). Consequently, the probability in \eqref{opnormbound2} is at most
\begin{align}
&\sum_{i=1}^N\sum_{1\le j_1<j_2 \le N} \PP(|a_{i, j_1}|>N^{\frac{1}{2}-\eta}, |a_{i, j_2}|>N^{\frac{1}{2}-\eta})+\sum_{j=1}^N\sum_{1\le i_1<i_2 \le N} \PP(|a_{i_1, j}|>N^{\frac{1}{2}-\eta}, |a_{i_2, j}|>N^{\frac{1}{2}-\eta}) \\ \le & 
2N \binom{N}{2}\frac{\sup_{N\in\N^+, i, j\in[N]}(\E\{|a_{i, j}^{(N)}|^4\})^2}{N^{4-8\eta}} \le \frac{1}{N^{1-8\eta}} \sup_{N\in\N^+, i, j\in[N]}(\E\{|a_{i, j}^{(N)}|^4\})^2 \overset{N\to\infty}{\to}0.
\end{align}
We used the independence of the entries in each row or column and Markov's inequality.

3) To deal with $|A_{N}^{\le}-\E A_{N}^{\le}|_{\op}$, we will use Proposition \ref{protasi gia megalo trace}. Let 
\begin{align}\label{orismos tou trunctated variance}
\Tilde{A}_{N}&:=A_{N}^{\le}- \E A_{N}^{\le},\\
s_{i,j}^{(N), \leq}&:=
\E\{(\Tilde A_{N})_{i,j}^{2}\}.
    \end{align}
Proposition \ref{protasi gia megalo trace} applies to $\Tilde{A}_N$ because any element of the matrix, say $(\Tilde A_N)_{i, j}$, has zero mean and variance $s_{i,j}^{(N), \leq}\le \E\{(A_N^{\le})^2\}\le \E(A_N^2)=s_{i,j}^{(N)}\le 1$. Thus, if we denote by $\Tilde{M}_{N}(m)$ the terms \eqref{orismos kalon orwn} for $m \in [N]$ and for the matrix $\Tilde{A}_{N}$, we will have $\Tilde{M}_{N}(m)\le {M}_{N}(m)$ for all $m\in \N$, and Proposition \ref{protasi gia megalo trace}, gives that for any $1\le k<N$,
\begin{equation}\E \tr(\Tilde{A}_{N}^{2k})\leq M_{N}(k)+ \sum_{s=1}^{k} (4k^{5})^{2k-2s} \left(2N^{\frac{1}{2}-\eta}\right)^{2k-2s} \sum_{t=1}^{(s+1)\wedge k}(4k^4)^{4(s+1-t)}M_{N}(t-1)\end{equation}
Now fix $C_1>0$, its value will be determined in \eqref{C1Choice} below. For $1\le k\le C_1 \log N$,
\begin{align} \label{TraceOfTrunc}\E \tr(\Tilde{A}_{N}^{2k})\leq C_2 N^{k+1} R^{2k} + C_2 \sum_{s=1}^{k} (4k^{5})^{2k-2s} \left(2N^{\frac{1}{2}-\eta}\right)^{2k-2s} \sum_{t=1}^{(s+1)\wedge k}(4k^4)^{4(s+1-t)} N^t R^{2(t-1)}.
\end{align}
 Next, we focus on the second summand in the right hand side of the previous inequality for $N$ large enough. In the sum in $t$ we factor out $(4k^4)^{4(s+1)} R^{-2}$, and in the resulting sum of geometric progression with ratio $a$ larger than 1 we use the bound $a+a^2+\cdots+a^{(s+1)\wedge k}\le k a^{(s+1)\wedge k}$. Thus the sum in \eqref{TraceOfTrunc} is bounded by
    \begin{align}
    &C_2 \frac{k}{R^2}\sum_{s=1}^{k} (4k^5)^{2k-2s} \left(2N^{\frac{1}{2}-\eta}\right)^{2k-2s} (4k^4)^{4(s+1)}\left(\frac{N R^2}{(4k^{4})^4}\right)^{(s+1)\wedge k}\\
&=2^8C_2 k^{17}N^k(R^{2})^{k-1}+C_2 \frac{k}{R^{2}} (N R^2)^{k+1} \sum_{s=1}^{k-1}\left(\frac{4 (4k^5)^2}{N^{2\eta} R^2}\right)^{k-s} \label{secondLine}\\
&\le 2^8 C_2 k^{17}N^{k}(R^{2})^{k-1}+2^7 C_2 k^{11} (R^2)^{(k-1)} N^{k+1-2\eta}\le 2^9 C_2 k^{17} N^{k+1-2\eta} R^{2k-2}
       \end{align}
[in summing the geometric series in \eqref{secondLine}, we used the bound $c+c^2+\cdots+c^r\le 2c$ if $0\le c<1/2$]. Thus, returning to \eqref{TraceOfTrunc},
\begin{equation}\label{ExpTraceBound}
    \E \tr (\Tilde{A}_{N}^{2k})\leq N^{k+1} R^{2k}\{1+o(1)\}^{2k}.
    \end{equation}
with $o(1)$ depending on $R, C_2, \eta$. \\
Fix $\epsilon>0$, pick 
\begin{equation} \label{C1Choice}
C_1> \frac{2+\epsilon}{\log(1+\epsilon)},
\end{equation}
and apply the above for $k:=[C_1 \log N]$. Relation \eqref{ExpTraceBound} implies
\begin{equation} \label{boundonthetruncSR}
\begin{aligned}
\mathbb{P}\left(\frac{|\Tilde{A}_{N}|_{\op}}{\sqrt{N}}\geq R(1+\epsilon)\right) &\le \mathbb{P}\left( \frac{|\Tilde{A}_{N}|^{2k}_{\op}}{N^{k}}\geq R^{2k}(1+\epsilon)^{2k}\right)\leq \frac{1}{R^{2k}(1+\epsilon)^{2k}} \frac{1}{N^{k}} \E|\Tilde{A}_{N}|_{\op}^{2k} \\
&\leq N \left(\frac{1+o(1)}{1+\epsilon}\right)^{2k}=O\left(\frac{1}{N^{1+\epsilon}}\right),
\end{aligned}
\end{equation}
for any $N$ large enough. The last equality is true because of the choice of $k$ and $C_1$.
\end{proof}

A tool for proving almost sure convergence of the sequence $|A_{N}|_{\op}/\sqrt{N}$ is the following lemma.

\begin{lem} \label{AsConvLemma}
Let $(A_{N})_{N\in\N^+}$ be a sequence of matrices, $A_N$ is $N\times N$, and $R>0$ so that the sequence satisfies Assumption \ref{Genikes ypotheseis}(a), condition $\Sigma(R)$, and Assumption \ref{assumfora.s.}. Then 
\begin{equation}
  \limsup_{N\to\infty} \frac{|A_{N}|_{\op}}{\sqrt{N}}\leq R   \, \text{ a.s.}
\end{equation}
\end{lem}

\begin{proof}
Pick $\eta\in(0, 1/8)$, its exact value will be determined below, and define the matrices $A_N^{\le}, \E A_N^{\le}$ as in the proof of Proposition \ref{PropSRCondition}. The proof will be accomplished once we show that
\begin{align}
    &\limsup_N\frac{|A_N^{\le}|_{\op}}{\sqrt{N}}\leq R, \ \  \text{ a.s., and} \label{AsTruncated} \\
    &\PP \left(A_{N}\neq A^{\leq}_{N} \text{ for infinitely many } N  \right)=0. \label{FinallyEqual}
\end{align}

\textsc{Proof of \eqref{AsTruncated}.} Since $\Sigma(R)$ holds for the sequence $(A_N)_{N\ge1}$, the proof of Proposition \ref{PropSRCondition} (the part with heading 3) shows that 
\begin{equation}
 \limsup_{N}\frac{|A^{\leq }_{N}-\E A_{N}^{\leq }|_{\op}}{\sqrt{N}}\leq  \mu_{\infty} \ \text{  a.s. }
\end{equation}
because the upper bound in \eqref{boundonthetruncSR} is summable with respect to $N$.
In the same proof it is shown that 
$$\limsup_{N}\frac{|\E A_{N}^{\leq }|_{\op}}{\sqrt{N}}=0   \text{  a.s. }$$
Using these two facts and the triangle inequality we get \eqref{AsTruncated}.

\textsc{Proof of \eqref{FinallyEqual}.} Let $X$ be the random variable that stochastically dominates the entries of $A_{N}$ in the sense of \eqref{stochdom}. Let $X_{N}$ be a sequence of symmetric random matrices after an appropriate coupling such that for all $N\in \N$ and $i,j\in [N]$ it is true that
     \begin{align}\label{anisotita stoch dom}
         |a^{(N)}_{i,j}|\leq |(X_{N})_{i,j}|
         \end{align}
     and the entries of $X_{N}$ are independent up to symmetry and all following the same law as $X.$ It is an easy exercise to show that for any $a, c>1$ and $Y$ real valued random variable we have 
     \begin{equation}
    \sum_{k=1}^\infty a^k \PP(|Y|\ge c^k)\le \frac{1}{a-1}\E\big\{|Y|^{\frac{\log a}{log c}}\big\}.
     \end{equation}
Using this inequality and the fact that the random variable $X$ has finite $4+\delta$ moment, we get that all $\eta\le\frac{\delta}{\delta+4}$ satisfy  
     \begin{align}\label{sumfinitea.s.}
         \sum_{m=1}^{\infty} 2^{2m}\PP(|X|\geq 2^{\frac{m}{2}(1-\eta)})<\infty.
     \end{align}    
Thus, picking in the beginning of the proof an arbitrary $\eta$ with $0<\eta<(1/8)\wedge(\delta/(4+\delta))$, we have
     \begin{align} & \label{i.o for A}
        \PP \left(A_{N}\neq A^{\leq}_{N} \text{ for infinitely many } N  \right) 
         = \PP \left(\text{ for infinitely many $N$ there are }  i, j\in [N]: |a_{i, j}^{(N)}|> C N^{\frac{1}{2}-\eta}  \right)
         \\ \label{i.o for X} &\le 
           \PP \left( \text{ for infinitely many $N$ there are }  i, j\in [N]: |X_{i, j}^{(N)}|> C N^{\frac{1}{2}-\eta}  \right)=\PP \left(X_{N}\neq X^{\leq}_{N} \text{ for infinitely many } N   \right).
         \end{align}
 In the second line, the inequality is a consequence of \eqref{anisotita stoch dom}, and the matrix $X^{\leq }_{N}$ is the matrix whose $(i, j)$ element is $(X_{N})_{i, j} \mathbf{1}_{|(X_N)_{i, j}|\le CN^{\frac{1}{2}-\eta}}$. The convergence of the series in \eqref{sumfinitea.s.}. implies that the probability in the right hand side of \eqref{i.o for X} is $0$ (see \cite{bai2010spectral}, pages 94 and 95) and finishes the proof of \eqref{FinallyEqual}.    
\end{proof}

\section{Proof of Theorem \ref{to theorima}}
The convergence $\mu_{A_N/\sqrt{N}}\Rightarrow \mu$ in probability 
implies that
\begin{equation} \label{LowerBound}
    \liminf_{N}\frac{|A_{N}|_{\op}}{\sqrt{N}}\geq \mu_{\infty} \ \ \ \text{in probability},
\end{equation}
that is, for all $\epsilon>0, \lim_{N\to\infty}\PP(|A_{N}|_{\op}/\sqrt{N}<\mu_{\infty}-\epsilon)=0$.
So in order to prove Theorem \ref{to theorima} one needs to prove that
\begin{equation}\label{H anisotita gia to theorima limsup}
\limsup_{N} \frac{|A_{N}|_{\op}}{\sqrt{N}} \leq \mu_{\infty}
\end{equation}
in probability. By Proposition \ref{PropSRCondition}, it is enough to prove that condition $\Sigma(\mu_\infty)$ is satisfied.

We will prove condition $\Sigma(\mu_\infty)$ separately for each of the Assumptions \ref{Assumptions 1} and \ref{Assumption 2} in the next two lemmas.

\begin{lem}\label{protasi gia anisotita kalon orwn me apeiro norma}
    Let $(A_N)_{N\in\N^+}$ be a sequence of matrices that satisfies Assumptions  \ref{ConvergenceOnTrees}, \ref{Genikes ypotheseis}, and \ref{Assumptions 1}. Then for every $k,N \in \N^+$ such that $k<N$ it is true that
    \[M_{N}(k)\leq N^{k+1}\mu_{\infty}^{2k}.\]
In case $\mu_\infty>0$, the inequality is true (as equality) for $k=0$ also.

\end{lem}
\begin{proof}    
Fix $N,k \in \N^+$ with $k<N$ and a tree $T \in \mathbf{C}_{k}$. Then, for each $\mathbf{d}:=(d_1, d_2, \ldots, d_{k+1})\in \{-1, 0\}^{k+1}$ consider the function
    \[\phi_{\mathbf{d}}:[N]^{k+1}\rightarrow [2N]^{k+1}\]
    with
    \[\phi_{\mathbf{d}}\left(i_{1},i_{2},\cdots,i_{k+1}\right)=2\left(i_{1},i_{2},\cdots.,i_{k+1}\right)+(d_{1},d_{2},\cdots,d_{k+1})\]
    for all $i_1, i_2, \ldots, i_{k+1}\in [N]$. 
    Each $\phi_{\mathbf{d}}$ is one to one and, for different vectors $\mathbf{d},\mathbf{d'}\in \{-1,0\}^{k+1}$, the image of $\phi_{\textbf{d}}$ is disjoint from that of $\phi_{\textbf{d'}}$. If $G$ is a plane rooted tree whose vertices in order of appearance in a depth first search are $(i_1, i_2, \ldots, i_{k+1})\in [N]^{k+1}$, and $\phi_{\mathbf{d}}(i_1, i_2, \ldots, i_{k+1})=(j_1, j_2, \ldots, j_{k+1})$, we denote by $\phi_{\mathbf{d}}(G)$ the plane rooted tree with vertex set $\{j_1, j_2, \ldots, j_{k+1}\}$, root $j_1$, and edges $\{\{j_a, j_b\}:\{a, b\} \text{ is an edge of } G\}$. Note that if all coordinates of $\mathbf{i}\in[N]^{k+1}$ are different, the same is true for the coordinates of $\phi_{\mathbf{d}}(\textbf{i})$.

    Lastly, by assumption \ref{Assumptions 1}, for any $T \in \mathbf{C}_{k}, \mathbf{i}\in[N]^{2k}$ such that $G(\mathbf{i})\sim T$ and $\mathbf{d}\in\{-1, 0\}^{k+1}$, it is true that 
    \begin{equation}\label{anisotita gia diplasiasmo kai gia ena dentro}
    \prod_{\{i,j\} \in E(G(\mathbf{i}))} s^{(N)}_{i,j}\leq \prod_{\{i,j\} \in E(\phi_{\mathbf{d}}(G(\mathbf{i})))}s^{(2N)}_{i,j}.
    \end{equation}
    So if one sums over all possible trees in $\mathbf{C}_{k}$ and $d\in\{-1, 1\}^{k+1}$, \eqref{anisotita gia diplasiasmo kai gia ena dentro} implies that
    \begin{equation}\label{anisotita gia diplasiasmo ropwn }
    2^{k+1}M_{N}(k)= \sum_{\mathbf{d}\in\{-1, 0\}^{k+1}}\sum_{T \in \mathbf{C}_{k}}\sum_{\mathbf{i}\in [N]^{2k}: G(\mathbf{i})\sim T} \prod_{\{i,j\} \in E(G(\mathbf{i}))} s^{(N)}_{i,j}\leq \sum_{\mathbf{d}\in\{-1, 0\}^{k+1}} \sum_{T \in \mathbf{C}_{k}}\sum_{\mathbf{i}\in [N]^{2k}: G(\mathbf{i})\sim T} \prod_{\{i,j\} \in E(\phi_{\mathbf{d}}(G(\mathbf{i})))}s^{(2N)}_{i,j}\leq M_{2N}(k).
    \end{equation}
    By applying \eqref{anisotita gia diplasiasmo ropwn } inductively, one can prove that for fixed $N,k \in \N$ the sequence
    \[q_{m}:=M_{2^{m}N}(k)/(2^m N)^{k+1}, m\in\N\]
    is increasing in the variable $m$. So by  \eqref{assumptionchangedgraphon} it is true that
    \[\sup_{m}q_{m}=\lim_{m \rightarrow \infty}q_{m}=\int x^{2k}\, d\mu(x)\le \mu_\infty^{2k}.\]
In particular, $q_{0}\leq \mu_{\infty}^{2k}$, completing the proof.
\end{proof}

\begin{lem}\label{Protasi gia anisotita graphon}
    Suppose $(A_{N})_{N\in\N^+}$ is a sequence of matrices such that Assumptions \ref{Genikes ypotheseis},\ref{Assumption 2} hold. Then for each $C_1>0$ there is $C_2>0$ such that
\begin{equation} \label{EtrBound}
 M_{N}(k)\leq C_2 N^{k+1} \mu_{\infty}^{2k}
\end{equation}
for all $N\in\N^+$ and $1\le k \le C_1\log N$.
\end{lem}

\begin{proof}
Note that for $1\le k<N$,
\begin{equation} \label{GraphonsTrees}
    \frac{1}{N^{k+1}}M_N(k)\le\sum_{T\in \mathbf{C}_k} \int_{[0, 1]^{k+1}} \bigg(\prod_{\{i, j\}\in E(T)} W_{N}(x_i, x_j)\bigg) \, dx_1 dx_2\cdots dx_{k+1}=:\Xi_N(k).
\end{equation}
The inequality holds because the left hand side results if on the right hand side we restrict the domain of integration to the union of the sets $\prod_{r=1}^{k+1}((i_r-1)/N, i_r/N]$ where all $i_1, i_2, \ldots, i_{k+1}\in[N]$ are different. Thus, it is enough to show \eqref{EtrBound} with the left hand side replaced by $N^{k+1}\Xi_N(k)$.

Fix $T \in \mathbf{C}_{k}$ and enumerate the edges of $T$ in the order of first appearance during a depth first search algorithm. For $\{i,j\} \in E(T)$, let $\{i,j\}_{\text{ord}}$ be its enumeration. Then for any integer $l\in [0, k]$ define the following quantities.
\begin{equation}\mu^{(l)}_{N}(k,T)=\int_{[0.1]^{k+1}}\prod_{\{i, j\}\in E(T):\{ i,j\}_{\text{ord}}\leq l } W_{N}(x_i, x_j) \prod_{\{i, j\}\in E(T): \{i,j\}_{\text{ord}}\geq l+1} W(x_{i},x_{j}) dx_1 dx_2\cdots dx_{k+1}.
\end{equation}
Note that 
\begin{equation}
\sum_{T\in \B{C}_k}\mu_{N}^{(0)}(k,T)=\mu_{2k}, \quad \sum_{T\in \B{C}_k} \mu_{N}^{(k)}(k,T)=\Xi_N(k)
\end{equation}
Fix $D>0$. Since all the variances are uniformly bounded by $1$, Assumption \eqref{fast convergence assumption} implies that there exists some $N_{0}(D)$ and $C>0$ such that for $N\geq N_{0}(D)$ and any $1\le l \le k<N$ it is true that
\begin{equation} \label{IncrBound} |\mu_{N}^{(l)}(k,T)-\mu_{N}^{(l-1)}(k,T)| \leq \int_{[0,1]^{2}}|W_{N}(x,y)-W(x,y)|\, dx \, dy\leq C \frac{1}{N^{D}}.
\end{equation}
Consequently, since $|\B{C}_k|\le 2^{2k}$,  for $k<N$ we have  
\begin{align} \label{telescopicBound1} \left|\Xi_{N}(k)-\mu_{2k}\right|&\leq \sum_{T \in \mathbf{C}}\Big| \sum_{l=1}^k\{\mu_{N}^{(l)}(k,T)-\mu_{N}^{(l-1)}(k,T)\}\Big| \leq \sum_{T \in \mathbf{C}_{k}}\sum_{l=1}^k|\mu^{(l)}_{N}(k,T)-\mu^{(l-1)}_{N}(k,T)|\leq \frac{C k 2^{2k}}{N^{D}}. 
\end{align}
Pick any $D>-2C_1\log(\mu_\infty/2)$. Then there is $N_0'\in\N^+, N_0'>N_0(D)$ such that $C k 2^{2k}/N^{D}\le \mu_\infty^{2k}$ for all $N>N_0'$ and $1\le k \le C_1 \log N$. And since $\mu_{2k}\leq \mu_{\infty}^{2k}$, we will have $\Xi_N(k)\leq 2 \mu_{\infty}^{2k}$ for the same $N$ and $k$. If we choose a constant  $C_2\ge 2$ so that \eqref{EtrBound} is satisfied for $N\in [N_0']$ and $1\le k\le  C_1\log N$, then we will have \eqref{EtrBound} for all $N, k$ claimed. 
\end{proof}

 \subsection{Proof of almost sure convergence under the additional Assumption \ref{assumfora.s.} }
 The convergence in probability that we have proven so far gives 
 \begin{align}
       \liminf_{N \to \infty} \frac{|A_N|_{\op}}{\sqrt{N}}\ge\mu_{\infty} \text{ a.s. }
    \end{align}
The opposite inequality follows from Lemma \ref{AsConvLemma} whose assumptions are satisfied, with $R=\mu_\infty$, because, under both scenarios of the Theorem, assumption $\Sigma(\mu_\infty)$ holds.

\section{Proof of Theorem \ref{Theorima gia sxedon kalous pinakes}}
The plan is to write the matrix $A_N$ as $A_N^{(1)}+A_N^{(2)}$ so that 
for the sequence $\{A_N^{(1)}\}_{N\ge1}$ we can apply Theorem \ref{to theorima}  while for $\{A_N^{(2)}/\sqrt{N}\}_{N\ge1}$ the operator norm will tend to zero.

Let
\begin{align}     
\mathcal{D}_N&:=\{(i, j)\in [N]^2: \text{ there exists } m \in [d_N]: (i, j)\in (\mathcal{B}^{(N)}_{m})^{\circ} \}.
\end{align} 
Then define the matrices
\begin{align}
\{A_{N}^{(1)}\}_{i,j}&:=\mathbf{1}_{(i,j) \in \mathcal{D}_{N}} \{A_{N}\}_{i,j}, \\ \{A_{N}^{(2)}\}_{i,j}&=\mathbf{1}_{(i, j)\notin \mathcal{D}_N}\{A_{N}\}_{i,j}.
\end{align}
\textbf{}
The proof follows from the triangle inequality for the operator norm and the following two statements, which we are going to prove next.
\begin{align} 
\label{oi liges kakes grammes}
\lim_{N \to \infty}\frac{|A^{(2)}_{N}|_{\op}}{\sqrt{N}}&=0\quad \text{in probability}. \\
\lim_{N\to\infty} \frac{|A_N^{(1)}|_{\op}}{\sqrt{N}}&=\mu_\infty.  \label{mainMatrix}
\end{align}

        \noindent \textsc{Proof of \eqref{oi liges kakes grammes}}. 
        For any $k \in \N $ denote by $M^{(2)}_N(k)$ the quantity \eqref{orismos kalon orwn} but with the role of $A_N$ played by $A^{(2)}_N$, i.e., $s_{i, j}^{(N)}$ is replaced by  $s_{i, j}^{(N)}\textbf{1}_{(i,j) \notin \mathcal{D}_N}$
     By Proposition \ref{PropSRCondition} it is sufficient to prove that for any constant $C_1>0$ it is true that for any $k\leq C_1 \log N$,
        \begin{align}\label{M^(2) bound}
            M^{(2)}_N(k) \leq N (8 d_N)^k.
        \end{align}
This is true because each product in \eqref{orismos kalon orwn} is at most 1, then the inner sum has at most $N(2d_N)^k$ non zero terms [there are  $N$ choices for $i_1$, and then, for each choice of $i_1$ there are at most $2d_N$ choices for $i_2$ that have $s_{i_1, i_2}^{(N)}\ne0$ due to condition (c) of Definition \ref{defngenvar}, and the same restriction holds for $i_3, \ldots, i_{k+1}$] and the outer sum has $|\mathbf{C}_k|\le 4^k$ terms. 

\smallskip

        \noindent \textsc{Proof of \eqref{mainMatrix}}. We will show that Theorem \ref{to theorima} can be applied to the sequence $\{A_N^{(1)}\}_{N\ge1}$. First we  prove that
        \begin{equation} \label{GoodPieceConv}
            \mu_{A_N^{(1)}/\sqrt{N}} \Rightarrow \mu \text{ in probability as } N\to\infty.
        \end{equation}
        As remarked after relation \eqref{assumptionchangedgraphon}, 
        $\mu_{A_N/\sqrt{N}} \Rightarrow \mu \text{ in probability as } N\to\infty.$ Then, from a well known inequality (Corollary A.41 in \cite{bai2010spectral}), the Levy distance between $\mu_{A_N/\sqrt{N}}$ and $\mu_{A_N^{(1)}/\sqrt{N}}$ is bounded as follows.
\begin{align}
L^3(\mu_{A_N/\sqrt{N}}, \mu_{A_N^{(1)}/\sqrt{N}}) \le \frac{1}{N}\text{tr}\bigg\{\bigg(\frac{1}{\sqrt{N}} A_N-\frac{1}{\sqrt{N}} A_N^{(1)}\bigg)^2\bigg\}=\frac{1}{N^2}\sum_{i, j\in [N]} \{(A_N^{(2)})_{i, j}\}^2.    
\end{align}
The expectation of the rightmost quantity is at most $N^{-2} N 2 d_N$ (since each row of $A_N^{(2)}$ has at most $2d_N$ elements that are not identically zero random variables and these random variables have second moment at most 1), which tends to 0 as $N\to0$ because of the assumption on $d_N$.

        Then the sequence $\{A_N^{(1)}\}_{N\ge1}$ satisfies: \\
        $\bullet$ Assumption \ref{ConvergenceOnTrees} with the same measure as $\{A_N\}_{N\ge1}$. This follows from Lemma \ref{BreakLemma}. Assumption (c) of that lemma is satisfied because of \eqref{GoodPieceConv}. \\
        $\bullet$ Assumption \ref{Genikes ypotheseis}, this is clear, \\   
        $\bullet$ Assumption \ref{Assumptions 1}. Indeed, fix $(i,j)\in [N]^2$. If $(i,j)\in \mathcal{D}_{N}$, there exists some $m\in[d_N]$ such that 
        \[\{(i+d_1,j+d_2): d_1, d_2\in\{-1, 0, 1\}\}\subseteq \mathcal{B}^{(N)}_{m}.\]
        Then  Assumption \eqref{growth of interior of convex sets} implies that there exists some $f\in [d_{2N}]$ such that
\[\{(2i+d_1, 2j+d_2): d_1, d_2\in\{-2, 0, 2\}\}\subseteq\mathcal{B}^{(2N)}_{f}.\]
But since $\mathcal{B}^{(2N)}_{f}$ is axially convex (see before Definition \ref{defngenvar}), one can conclude that
\[\{(2i+d_1, 2j+d_2): d_1, d_2\in\{-2, -1, 0, 1, 2\}\}\subseteq\mathcal{B}^{(2N)}_{f}.\]
Now since $(k, \ell)\mapsto s_{k, \ell}^{(2N)}$ is constant in $\mathcal{B}^{(2N)}_{f}$ [see \eqref{defnofvariancesconvex}] and we assumed \eqref{extrassuconvex}, our claim follows.

Thus, all the Assumptions of Theorem \ref{to theorima} hold for $A_N^{(1)}$, and hence \eqref{mainMatrix} holds. \\
\textit{Almost sure convergence under the additional Assumption \ref{assumfora.s.}}.  Using Lemma \ref{AsConvLemma}, we will prove that 
\begin{align}\label{ineqfortheorA^1a.s.}
                 \limsup_{N\to \infty}\frac{|A_N^{(1)}|_{\op}}{\sqrt{N}} & \leq \mu_{\infty} \ \  \text{ a.s. } \\
             \label{ineqfortheorA^2a.s.}
    \limsup_{N\to \infty}\frac{|A_N^{(2)}|_{\op}}{\sqrt{N}}&\le \epsilon \ \  \text{a.s.  for any }\epsilon>0.
\end{align}
And these are enough to prove our claim.

Notice that the validity of Assumptions \ref{Genikes ypotheseis}(a) and  \ref{assumfora.s.} for the sequence $(A_{N})_{N\in\N^+}$ implies the validity of the same assumptions for the sequences $(A^{(1)}_{N})_{N\in\N^+}$ and $(A^{(2)}_N)_{N\in\N^+}$. 
As was mentioned above, the sequence $\{A_N^{(1)}\}_{N\ge1}$ satisfies Assumption \ref{ConvergenceOnTrees} with the same measure as $\{A_N\}_{N\ge1}$. And then Lemma \ref{protasi gia anisotita kalon orwn me apeiro norma} implies that the sequence $(A^{(1)}_N)_{N\in\N^+}$ satisfies condition $\Sigma(\mu_{\infty})$, while \eqref{M^(2) bound} and $\lim_{N\to\infty} d_N/n=0$ imply that, for any $\epsilon>0$, the sequence $(A^{(2)}_N)_{N\in\N^+}$ satisfies condition $\Sigma(\epsilon)$. Thus, Lemma \ref{AsConvLemma} applies and gives the desired inequalities. 
        
\section{Step function profile. Proof of Theorem \ref{theoremstepfunction}}
\begin{proof}[Proof of Theorem \ref{theoremstepfunction}] The inequality $\liminf_{N\to\infty}|A_N|_{\op}/\sqrt{N}\ge \mu_\infty^{\sigma}$ almost surely is justified with the same argument as \eqref{LowerBound} with the only difference that here we have $\mu_{A_N/\sqrt{N}}\Rightarrow \mu^{\sigma}$ a.s., and so the inequality will be true in the a.s. sense. 

For the reverse inequality, we will apply Lemma \ref{AsConvLemma}. To check Assumptions \ref{Genikes ypotheseis}(a) and \ref{assumfora.s.}, required by that lemma, note that the $(i, j)$ element of $A_N$ is of the form $\sigma_{p, q} X_0'$ for a constant $\sigma_{p. q}\in[0, 1]$ and $X_0'\overset{d}{=}X_0$, and clearly $X_0$ can play the role of $X$ in relation \eqref{stochdom}. We will prove that $(A_N)_{N\in\N^+}$ satisfies condition $\Sigma(c)$ for all $c>\mu_\infty$, and this will finish the proof.

Define the matrix $\hat \Sigma_N\in\R^{N\times N}$ by
\begin{equation} \label{SigmaHatDef}
    (\hat \Sigma_N)_{i,j}=\sigma_{p, q} \text{\qquad if } \begin{array}{l} a_{p-1}+(1/N)\le i/N <a_p \text{ and }\\ a_{q-1}+(1/N)\le j/N < a_q, \end{array}
\end{equation}
and $\hat A_N:=\hat \Sigma_N\odot A'_N$.

Also, let 
\begin{equation}
\varepsilon_N:=\max\bigg\{\Big|\frac{R_p^{(N)}}{N}-\alpha_p\Big|:p\in[m-1]\bigg\}.\end{equation}
By Definition \ref{defn rmt step function}, it holds $\lim_{N\to\infty} \varepsilon_N=0$.

Claim 1: a) With probability one, $\hat A_N/\sqrt{N}$ has the same limiting ESD as  $A_N/\sqrt{N}$.

    b) For $\hat A_N$, Lemma \ref{protasi gia anisotita kalon orwn me apeiro norma} applies. 
    
Consequently, 
\begin{equation} \label{GMBound}
    \hat M_N(k)\le N^{k+1} (\mu_\infty^\sigma)^{2k}
\end{equation}
for all $1\le k<N$.

Proof of Claim 1:

a) This is true because by Theorem A.43 in the Appendix A of \cite{bai2010spectral}, the Kolmogorov distance between $\mu_{A_N/\sqrt{N}}$ and $\mu_{\hat A_N/\sqrt{N}}$ is at most
\begin{align}
  \frac{1}{N}\operatorname{rank} (A_N-\hat A_N)\leq  \frac{m}{N} \max_{p \in [m]}( \max \{ R_p, N a_p\}- \min\{R_p, N a_p\})= m \epsilon_N \overset{N \to \infty}{\to} 0.
\end{align}

b) Assumption \ref{Genikes ypotheseis} is satisfied because $\E(|X_0|^{4+\delta})$ and $\sigma_{p, q}\le 1$ for all $p, q\in[m]$. To show that Assumption \ref{ConvergenceOnTrees} is satisfied, we repeat the argument just before the statement of the Theorem. For the sequence $(\hat A_N)_{N\in \N^+}$, the corresponding $W_N(x, y)$, as $N\to\infty$, converges to $\sigma^2(x, y)$ for almost all $(x, y)\in [0, 1]^2$. Assumption \ref{Assumptions 1} is satisfied because if for some $i, j$ we have $\text{Var}[(\hat A_N)_{i, j}]>0$, then this equals $\sigma^2_{p, q}$ for the unique $p, q$ as in \eqref{SigmaHatDef}. Then 
\begin{align}&\frac{2i-1}{2N}\in[a_{p-1}+\frac{1}{2N}, a_p-\frac{1}{2N}), 
 \frac{2i}{2N}\in[a_{p-1}+\frac{1}{N}, a_p),\\
&\frac{2j-1}{2N}\in[a_{q-1}+\frac{1}{2N}, a_q-\frac{1}{2N}), \frac{2j}{2N}\in[a_{q-1}+\frac{1}{N}, a_q).
\end{align}
Thus, \eqref{VarInequalities} holds as equality.

\medskip

Claim 2: There is $\theta\in(0, \infty)$ so that $M_N(k)\le e^{\theta(k+1)\varepsilon_N} \hat M_N(k)$ for all $k<N$.

Proof of Claim 2: Define the following sets of indices.
\begin{align}
\Delta_p^{(N)}&:=I_p^{(N)}\cap[a_{p-1}N+1, a_p N),\\
\Delta^{(N)}&:=\cup_{p=1}^m   \Delta^{(N)}.
\end{align}
When $p=m$, the interval in the intersection becomes closed on the right also.
Then 
\begin{align}
M_N(k)&\le \hat M_N(k)+\sum_{\emptyset\ne J\subset [k+1]}  \sum_{T \in \mathbf{C}_k}\sum_{i_1\cdots i_{k+1}} \mathbf{1}\left(i_l \notin \Delta^{(N)} \text{ if } l \in J  \text{ and } i_l \in \Delta^{(N)} \text{ if } l \notin J  \right) \prod_{\{i,j\} \in E(T)} s_{i,j}^{(N)}\\
&=\hat M_N(k)+\sum_{\emptyset\ne J\subset [k+1]}  \sum_{T \in \mathbf{C}_k}\sum_{m_1, \ldots, m_{k+1}\in [m]^{k+1}} a(T, J, m_1, m_1, \ldots, m_{k+1}) \label{BadIndSplit}
\end{align}
where
\begin{equation}\label{thirdiineqforstep}
a(T,J,m_1\cdots m_{k+1}):=      \sum_{i_1\in I_{m_1}^{(N)}, \cdots i_{k+1}\in I_{m_{k+1}}^{(N)}} \mathbf{1}(
i_l \notin \Delta^{(N)} \text{ if } l \in J, i_l \in \Delta^{(N)} \text{ if } l \notin J, (i_\ell)_{\ell\in[k+1]} \text{distinct}) \prod_{\{i,j\} \in E(T)} s_{i,j}^{(N)}.
\end{equation}
Note that
\begin{align}\label{M^(1) rewritten}
    \sum_{T\in \mathbf{C}_k}\sum_{m_1,m_2\cdots m_{k+1}\in [m]^{k+1}} a(T,\emptyset,m_1\cdots m_{k+1})\le \hat M_N(k).
\end{align}
We will show that for some constant $\theta=\theta(I_1, I_2, \ldots, I_m)\in (0, \infty)$ we have
\begin{align}\label{fourthineqforstep}
    a(T,J,m_1\cdots m_{k+1})\leq (\theta \varepsilon_N)^{|J|}a(T,\emptyset,m_1,\cdots m_{k+1}).
\end{align}

In the definition of $a(T, J, m_1, m_2, \ldots, m_{k+1})$, the product is common to all summands [recall the rectangles of constancy of the map $(i, j)\mapsto s_{i, j}^{(N)}$]. We write $a(T,J,m_1\cdots m_{k+1})$ and $a(T,\emptyset, m_1, \ldots m_{k+1})$ as

\begin{align} \label{doubleSumJ}
\sum_{i_\ell\in I_{m_\ell}^{(N)} \text{ for } \ell\notin J} \mathbf{1}(i_l \in \Delta^{(N)} \text{ for } l \notin J, (i_\ell)_{\ell\in[k+1]\sm J} \text{ distinct}) \sum_{i_\ell\in I_{m_\ell}^{(N)} \text{ for all } \ell\in J}\mathbf{1}(
i_l \notin \Delta^{(N)} \text{ for all } l \in J, (i_\ell)_{\ell\in[k+1]} \text{distinct}) 
\prod_{\{i,j\} \in E(T)} s_{i,j}^{(N)},\\ \label{doubleSum0}
\sum_{i_\ell\in I_{m_\ell}^{(N)} \text{ for } \ell\notin J} \mathbf{1}(i_l \in \Delta^{(N)} \text{ for } l \notin J, (i_\ell)_{\ell\in[k+1]\sm J} \text{ distinct}) \sum_{i_\ell\in I_{m_\ell}^{(N)} \text{ for all } \ell\in J}\mathbf{1}(
i_l \in \Delta^{(N)} \text{ for all } l \in J, (i_\ell)_{\ell\in[k+1]} \text{distinct}) 
\prod_{\{i,j\} \in E(T)} s_{i,j}^{(N)}
\end{align}
We will compare the inner sums in the two expressions.
 Notice that there are $C_1, C_2>0$ that depend on $a_1, a_2, \ldots, a_m$ only so that 
\begin{align}
|\Delta_p^{(N)}|&\ge C_1 N,\\
|I_p^{(N)}\sm \Delta^{(N)}|&\le C_2 \varepsilon_N N
\end{align}
for all $p\in[m]$. For each fixed collection $(i_\ell)_{\ell\notin J}$, the inner sum in \eqref{doubleSumJ} is at most $(C_2\varepsilon_N N)^{|J|}$ while the inner sum in \eqref{doubleSum0} is at least $(C_1N/2)^{|J|}$. The ratio of the first over the second bound is $(2C_2\varepsilon_N/C_1)^{|J|}$. Thus, we get \eqref{fourthineqforstep} with $\theta:=2C_2/C_1$.

Taking into account \eqref{fourthineqforstep} and \eqref{M^(1) rewritten}, we get that the second summand in \eqref{BadIndSplit} is bounded above by
\begin{align} \sum_{t=1}^{k+1}\sum_{J\subset [k+1]: |J|=t} \hat M_{N}(k) (\theta \varepsilon_N)^t=M^{(1)}_N(k) \sum_{t=1}^{k+1} \binom{k+1}{t} (\theta \varepsilon_N)^t=\hat M_N(k)\left\{\left(1+\theta \varepsilon_N\right)^{k+1}-1\right\}
\end{align}

Consequently, $M_N(k)\le (1+\theta \varepsilon_N)^{k+1} \hat M_N(k)\le e^{\theta (k+1) \varepsilon_N} \hat M_N(k)$, and this proves Claim 2.

Now, combining this with \eqref{GMBound}, we get that condition $\Sigma((1+\varepsilon)\mu_\infty)$ is satisfied for each $\varepsilon>0$. 
\end{proof}

        \section{An approximation result and proof of Corollary \ref{theoremconfunction}}  \label{ContModelSection}
        
        \begin{prop}\label{symperasma gia proseggiseis pinakon}
    Let $(A_N)_{N\in\N^+}$ be a sequence of symmetric random matrices, $A_N$ of dimension $N\times N$, of the form
    \begin{equation}
    A_{N}=\Sigma_{N}\odot A'_{N},
    \end{equation}
     where $\Sigma_{N}\in[0, \infty)^{N\times N}$ and $A'_{N}$ is a random $N\times N$ symmetric matrix with independent entries (up to symmetry) all with zero mean and unit variance. 
     
     For every $n\in \N^+$ consider a sequence $(\Sigma^{(n)}_{N})_{N\in\N^+}$ of matrices, with $\Sigma^{(n)}_{N}\in [0, \infty)^{N\times N}$, and define
     \begin{equation} A^{(n)}_{N}:=\Sigma^{(n)}_{N}\odot A'_{N} \quad \text{ for each } N\in \N^+. \end{equation}
(a) Assume that

\begin{itemize}
    \item[(i)]  the sequence $(A_N')_{N\in\N^+}$ satisfies Assumption \ref{Genikes ypotheseis},

    \item[(ii)] for each $n\in\N^+$ it holds
\begin{equation} \label{anisotita kales prosegiseis}
   \lim_{N \rightarrow \infty} \frac{|A^{(n)}_N|_{\op}}{\sqrt{N}} =\mu_\infty^{(n)} \text{ in probability, } 
\end{equation}
where $\mu_\infty^{(n)}$ is a finite constant,

\item[(iii)] \begin{equation}\lim_{n \to \infty}\mu^{(n)}_{\infty}=:\mu_{\infty}\in \mathbb{R},\label{EdgeConvergence} 
\end{equation}
\item[(iv)]\begin{equation}
\lim_{n\to\infty}\limsup_{N\to\infty}|\Sigma_{N}-\Sigma^{(n)}_N|_{\max}=0.\label{ApproxVariance} \end{equation}
\end{itemize}
Then 
\begin{equation} \label{ConvAfterApprox}
\lim_{N \rightarrow \infty}\frac{|A_{N}|_{\op}}{\sqrt{N}}=\mu_{\infty} \qquad \text{in probability.}
\end{equation}
(b) Assume that, in addition to the assumptions of (a), the convergence in \eqref{anisotita kales prosegiseis} holds in the a.s. sense and Assumption \ref{assumfora.s.} holds for the sequence $(A_N')_{N\in\N^+}$. Then the limit in \eqref{ConvAfterApprox} holds in the a.s. sense.
\end{prop}
\begin{proof} (a) Fix $\epsilon\in(0, 1/2)$ and $n_{0}$ large enough such that for every $n\geq n_{0}$ it is true that 
$$|\mu_{\infty}-\mu_{\infty}^{(n)}|\leq \epsilon \text{ and }\limsup_{N\to\infty}|\Sigma_{N}-\Sigma^{(n)}_N|_{\max}<\epsilon.$$
Fix an $n\ge n_0$. There is an $N_0=N_0(n)\in \N^+$ so that $|\Sigma_{N}-\Sigma^{(n)}_N|_{\max}<\epsilon^2$ for all $N\ge N_0$. 
Then for $N\ge N_0$ we have
\begin{align}
\PP\left(\left|\frac{|A_{N}|_{\op}}{\sqrt{N}}-\mu_{\infty}\right|\ge 5\epsilon\right) \leq \PP\left(\left|\frac{|A_{N}|_{\op}}{\sqrt{N}}-\mu^{(n)}_{\infty}\right|\ge 4\epsilon\right) \leq              \PP\left(\frac{|A_{N}-A^{(n)}_{N}|_{\op}}{\sqrt{N}}\geq 3\epsilon\right)+\PP\left(\left|\frac{|A^{(n)}_{N}|_{\op}}{\sqrt{N}}-\mu^{(n)}_{\infty}\right|\ge\epsilon\right). \label{pano fragma gia proseggisimous pinakes} 
\end{align}
The last term in \eqref{pano fragma gia proseggisimous pinakes} converges to zero as $N\to\infty$ due to \eqref{anisotita kales prosegiseis}. For the previous term we will apply Proposition \ref{PropSRCondition}. 
Notice that the sequence $(A_N-A_N^{(n)})_{N\in\N^+}$ satisfies 
\begin{itemize}
    \item Assumption \ref{Genikes ypotheseis} because $(A_N-A_N^{(n)})_{i, j}=((\Sigma_N)_{i, j}-(\Sigma_N^{(n)})_{i, j})(A'_N)_{i, j}$ and $|(A_N-A_N^{(n)})_{i, j}|\le |(A'_N)_{i, j}|$ (for all $N\in\N^+, i, j\in[N]$) and we assumed that $(A_N')_{N\in\N^+}$ satisfies Assumption \ref{Genikes ypotheseis}
    \item condition $\Sigma(2\varepsilon)$ because if, for $t\in \N^+$ with $ t<N$, we call $M'_N(t)$ the quantity defined in \eqref{orismos kalon orwn} for the matrix $A_{N}-A^{(n)}_{N}$, and note that the $(i, j)$ element of $A_{N}-A^{(n)}_{N}$ has mean zero and variance $\{(\Sigma_N)_{i, j}-(\Sigma_N^{(n)})_{i, j}\}^2$,  we obtain that 
            \begin{equation}\label{ineqA^n-A}M'_{N}(t)\leq N^{t+1}2^{2t}\big(|\Sigma_{N}-\Sigma^{(n)}_{N}|_{\max}\big)^{2t}<N^{t+1}(2 \epsilon)^{2t}.\end{equation}
\end{itemize}

 Since $3\epsilon>2\epsilon(1+\epsilon)$, Proposition \ref{PropSRCondition} implies that the penultimate term in \eqref{pano fragma gia proseggisimous pinakes} goes to zero as $N\to\infty$. 

\medskip 

(b) It is enough to prove that with probability 1 it holds $\varlimsup_{N \rightarrow \infty}\frac{|A_{N}|_{\op}}{\sqrt{N}}\le \mu_{\infty}$. Because of \eqref{anisotita kales prosegiseis} (holding a.s.) and  \eqref{EdgeConvergence}, it is enough to prove that for all $\epsilon>0$ and all $n$ large enough, with probability 1, it holds
\begin{equation}\varlimsup_{N \rightarrow \infty}\frac{|A_{N}-A_N^{(n)}|_{\op}}{\sqrt{N}}\le 2 \epsilon.\end{equation}
To prove this, we will apply Lemma \ref{AsConvLemma}.  Take $n_0$ so that for all $n\ge n_0$ it holds $\limsup_{N\to\infty}|\Sigma_{N}-\Sigma^{(n)}_N|_{\max}<\epsilon^2.$ Now fix $n\ge n_0$. There is $N_0=N_0(n)\in\N^+$ so that $|\Sigma_{N}-\Sigma^{(n)}_N|_{\max}<\epsilon^2$ for all $N\ge N_0$. Then the sequence $(A_N-A_N^{(n)})_{N\ge N_0}$ satisfies Assumption \ref{Genikes ypotheseis}(a) as we saw in part a) of the proposition, Assumption \ref{assumfora.s.} (because $A'_N$ does so and $|\Sigma_{N}-\Sigma^{(n)}_N|_{\max}<1$), and assumption $\Sigma(2\epsilon)$ because of \eqref{ineqA^n-A}. Then Lemma \ref{AsConvLemma} gives the desired inequality.
\end{proof}

        \begin{proof}[Proof of Corollary \ref{theoremconfunction}] We will apply Proposition \ref{symperasma gia proseggiseis pinakon}(b) for the sequence $(A_N)_{N\in\N^+}$. The sequence $(A_N')_{N\in\N^+}$ mentioned in that Proposition is exactly the sequence $(A_N')_{N\in\N^+}$ of relation \eqref{ContinuousModel} and it satisfies Assumption \ref{Genikes ypotheseis} because for it the discussion following Assumption \ref{Genikes ypotheseis} applies ($X_0$ has finite fourth moment). 
        
            For each $n\in\N^+$, we define the following obvious approximation to $\sigma$. 
\begin{equation} \label{Discretization}
\sigma^{(n)}(x,y):= n^2\int_{I_k}\int_{I_{\ell}}\sigma(a,b)\, da\, db  \text{ if } (x, y)\in I_k\times I_\ell \text{ for } k, \ell\in[n], 
\end{equation}
where $I_k:=[\frac{k-1}{n},\frac{k}{n})$ for $k\in[n-1]$ and $I_n:=[(n-1)/n, 1]$. Then, we define the matrices $\Sigma_N^{(n)}$ through the relation $(\Sigma_N^{(n)})_{i, j}:=\sigma^{(n)}(i/N, j/N)$.

For each $n \in \N^+$, the sequence of matrices $(\Sigma_{N}^{(n)}\odot A'_{N})_{N\ge1}$ satisfies the assumptions of Theorem \ref{theoremstepfunction}. Consequently, as $N\to\infty$, the sequence $(\mu_{A^{(n)}_N/\sqrt{N}})_{N\in\N^+}$ converges almost surely weakly to a symmetric measure, say $\mu^{(n)}$, with support contained in $[-\mu_\infty^{(n)}, \mu_\infty^{(n)}]$ and \eqref{anisotita kales prosegiseis} holds in the a.s. sense. In a claim below we prove that condition \eqref{EdgeConvergence} is satisfied. Finally, to check \eqref{ApproxVariance}, note that 
\begin{equation} \label{StepApproxCont} \left|(\Sigma_N)_{i, j}-(\Sigma_N^{(n)})_{i, j}\right|\le 
\left|(\Sigma_N)_{i, j}-\sigma(i/N, j/N)\right|+\left|\sigma(i/N, j/N)- \sigma^{(n)}(i/N, j/N)\right|.
\end{equation}
In the right hand side of the last inequality, the first term converges to zero as $N\to\infty$ due to \eqref{ContinuousVPApprox}, and the second term is at most the supremum norm of $\sigma-\sigma^{(n)}$, which goes to zero as $n\to\infty$ because $\sigma$ is uniformly continuous in $[0, 1]^2$. Thus, Proposition \ref{symperasma gia proseggiseis pinakon}(b) applies and completes the proof.  \\        
\textsc{Claim:} Condition \eqref{EdgeConvergence} is satisfied. \\
We modify the proof of Lemma 6.4 of \cite{husson2022large}.
Call $\mu$ the weak limit as $N\to\infty$ of $\mu_{A_N/\sqrt{N}}$, then $F_N, F_N^{(n)}$ the distribution function of $\mu_{A_N/\sqrt{N}}$ and $\mu_{A_N^{(n)}/\sqrt{N}}$ respectively, and $F, F^{(n)}$ the distribution function of $\mu, \mu^{(n)}$ respectively. Let 
\begin{align*}
\lambda_{N, 1}\le \lambda_{N, 2}\le \cdots \le \lambda_{N, N}, \\
    \lambda_{N, 1}^{(n)}\le \lambda_{N, 2}^{(n)}\le \cdots \le \lambda_{N, N}^{(n)}
\end{align*}
the eigenvalues of $A_N/\sqrt{N}, A_N^{(n)}/\sqrt{N}$ respectively. 

Let $\epsilon\in(0, 1/2)$. There is $n_0=n_0(\epsilon)$ so that for all $n\ge n_0$ it holds $\limsup_{N\to\infty} |\Sigma_N-\Sigma_N^{(n)}|_{\max}<\epsilon^2$. Take now $n\ge n_0$ fixed. There is $N_0=N_0(n)\in\N^+$ so that $|\Sigma_N-\Sigma_N^{(n)}|_{\max}<\epsilon^2$ for all $N\ge N_0$. As explained in the proof of Proposition \ref{symperasma gia proseggiseis pinakon}, $\lim_{N\to\infty}\PP\left(|A_{N}-A^{(n)}_{N}|_{\op}\geq 3\epsilon\sqrt{N}\right)=0$. There is sequence $(N_k)_{k\ge1}$ so that in a set $\Omega_\epsilon$ of probability 1, eventually for all $k$ we have  $|A_{N_k}-A^{(n)}_{N_k}|_{\op}<3\epsilon\sqrt{N_k}$. Since
$$\max_{i\in[N]}|\lambda_{N_k, i}^{(n)}-\lambda_{N_k, i}|\le |A_{N_k}-A^{(n)}_{N_k}|_{\op}/\sqrt{N_k},$$
in $\Omega_\epsilon$ (the inequality is true by Theorem A46 in \cite{bai2010spectral}),  we will have eventually for all $k\in\N^+$ that 
\begin{equation}
F_{N_k}^{(n)}(a-3\epsilon) \le F_{N_k}(a)\le F_{N_k}^{(n)}(a+3\epsilon) 
\end{equation}
for all $a\in\R$. From here, using the convergence as $N\to\infty$ of $F_N$ to $F$ and of $F_N^{(n)}$ to $F^{(n)}$, we have that for all $a\in\R$ it holds
\begin{equation}
F^{(n)}(a-3\epsilon) \le F(a)\le F^{(n)}(a+3\epsilon). 
\end{equation}
[First we get this for all $a$ outside a countable subset of $\R$ and then using the right continuity of $F, F^{(n)}$ we get it for all $a\in\R$.]
This implies that $|\mu_{\infty}^{(n)}-\mu_{\infty}|\le 3\epsilon$ and finishes the proof of the claim.
\end{proof}
\begin{rem} The above proof easily generalizes to the case that the function $\sigma$ is piecewise continuous in the following sense. There are $m\in\N^+$,  $0=a_0<a_1<\cdots<a_{m-1}<a_m=1$ so that letting $I_p:=[a_{p-1}, a_p)$ for $p=1, 2, \ldots, m-1$, and $I_m=[a_{m-1}, 1]$ the function $\sigma|I_p\times I_q$ is uniformly continuous for all $p, q\in[m]$ (i.e., when $\sigma$ extends continuously in the closure of each rectangle $I_p\times I_q$. Recall that to handle the last term in \eqref{StepApproxCont} all we needed was the uniform continuity of $\sigma$. 
\end{rem}

\section{Examples}
\subsection{Random Gram matrices} Let $(X_N)_{N\in\N^+}$ be a sequence of matrices so that $X_N$ is an $M(N)\times N$ matrix with independent, centered
entries with unit variance, and $M:\N^+\to\N^+$ a function with $\lim_{N\to\infty}\frac{M(N)}{N}=c\in(0, \infty)$. It is known
that the empirical spectral distribution of $XX^{T}$, after rescaling, converges to the Marchenko-Pastur law $\mu_{MP}$ \cite{marvcenko1967distribution}. Moreover, the convergence of the rescaled largest eigenvalue to the largest element of the support of $\mu_{MP}$ has  been established in \cite{YBK1988} under the assumption of finite fourth moment for the entries. However, some
applications in wireless communication require understanding the spectrum of 
$XX^{T}$, where X has
    a variance profile, see for example \cite{hachem2008clt} or \cite{couillet2011random}. Such matrices are called random Gram matrices. In this subsection, we establish the convergence of the largest eigenvalue of random Gram matrices to the largest element of the support of its limiting empirical spectral distribution for specific variance profiles. Firstly we give some definitions. 

\begin{defn}[Step function variance profile] \label{NSStepVP}
Consider
\begin{itemize}\item[a)] $m, n\in\N^+$ and numbers $\{\sigma_{p,q}\}_{p\in[m], q\in[n]}\in [0,\infty)^{mn}$. 
\item[b)]  For each $K\in\N^+$, two partitions $\{I_p^{(K)}\}_{p \in [m]}, \{J_p^{(K)}\}_{p \in [n]}$  of $[K]$ in $m$ and $n$ intervals respectively. The numbering of the intervals is such that $x<y$ whenever $x\in I_p^{(K)}, y\in I_q^{(K)}$ or $x\in J_p^{(K)}, y\in J_q^{(K)}$ with $p<q$. Let $LI_p^{(K)}$ and $RI_p^{(K)}$ be the left and right endpoint respectively of $I_p^{(K)}$ and similarly $LJ_p^{(K)}$ and $RJ_p^{(K)}$ for $J_p^{(K)}$.

\item[c)] Numbers $0=\alpha_0<\alpha_1<\cdots<\alpha_{m-1}<\alpha_m:=1$. We assume that $\lim_{M\to\infty}RI_p^{(M)}/M=\alpha_p$ for each $p\in[m]$.

\item[d)] Numbers $0=\beta_0<\beta_1<\cdots<\beta_{n-1}<\beta_n:=1$. We assume that $\lim_{N\to\infty}RJ_q^{(N)}/N=\beta_q$ for each $q\in[n]$.

\item[e)]$M:\N^+\to\N^+$ a function, 

\item[f)] A random variable $X_0$ with $\E(X_0)=0, \E(X_0^2)=1$. 
\end{itemize}
For each $M, N\in\N^+$, define the matrix $\Sigma_{M, N}\in\R^{M\times N}$ by $(\Sigma_{M, N})_{i, j}=\sigma_{p, q}$ if $i\in I_p^{(M)}, j\in I_q^{(N)}$, and let $\{A_{N}\}_{N \in \N^+}$ be the sequence of random matrices defined by
\begin{equation} A_N=\Sigma_{M(N), N}\odot A'_{M(N), N} \label{NSStepModel}
\end{equation}
where $A'_{M(N), N}$ is an $M(N)\times N$ matrix whose elements are independent random variables all with distribution the same as $X_0$. 
We say that $A_N$ in \eqref{NSStepModel} is a \emph{random matrix model whose variance profile is given by a step function.} 
\end{defn}

\begin{defn}[Continuous function variance profile] \label{NSContVP}
           For
           \begin{itemize}
           \item[a)] a continuous function $\sigma:[0, 1]^2\to[0, 1]$,
           \item[b)]$M:\N^+\to\N^+$ a function 
           \item[c)] a sequence $(\Sigma_{M(N), N})_{N\in\N^+}$ of matrices, $\Sigma_{M(N), N}\in[0, 1]^{M(N)\times N}$, with the property
           \begin{equation} \label{ContinuousVPApprox}
               \lim_{N\to\infty}\sup_{i\in[M(N)], j\in[N]}\left|(\Sigma_{M(N), N})_{i, j}-\sigma(i/M(N), j/N)\right|=0,
           \end{equation}
           \item[d)] a random variable $X_0$ with $\E(X_0)=0, \E(X_0^2)=1$,
           \end{itemize}
 consider the sequence $\{A_{N}\}_{N \in \N^+}$ of random matrices, $A_N\in\R^{M(N)\times N}$, defined by
\begin{equation} A_{N}=\Sigma_{M(N), N}\odot A'_{N} \label{NSContinuousModel}
\end{equation}
where the entries of $A'_N$ are independent random variables all with distribution the same as $X_0$. Then we say that $(A_N)_{N \in \N^+}$ is a \emph{random matrix model whose variance profile is given by a continuous function.} Again, we call $\sigma$ the variance profile. \end{defn}

\textsc{Symmetrization}
To study the eigenvalues of $A_N A_N^T$, where $A_N$ falls in one of the cases of the two last definitions, we use the trick of symmetrization. If $A$ is an $M\times N$ matrix, where $M, N\in \N^+$, we call symmetrization of $A$ the $(M+N)\times(M+N)$ symmetric matrix $\tilde A$ defined by
    \begin{align}\label{AN_symmetrization}
        \tilde A:=\begin{bmatrix}\mathbf{O}_{M,M} &  A \\ 
           A^T & \mathbf{O}_{N,N}\end{bmatrix}
           \end{align}
     where, for any $k, l\in\N^+$,  $\mathbf{O}_{k,l}$ denotes the $k\times l$ matrix with all of its  entries equal to $0$. The characteristic polynomials of $AA^T, \tilde A$ are connected through the relation
     \begin{equation}
         \lambda^M\det(\lambda \mathbb{I}_{M+N}-\tilde A)=\lambda^N\det(\lambda^2\mathbb{I}_M-A A^T)
     \end{equation}
     for all $\lambda\in\C$. Thus, in the case $M\le N$, if we call $(t_1, t_2, \ldots, t_{M+N})$ the eigenvalues of the symmetric matrix $\tilde A$, then the vector $(t_1^2, t_2^2, \ldots, t_{M+N}^2)$ contains twice each eigenvalue of $AA^T$ and $N-M$ times the eigenvalue 0 (multiple eigenvalues appear in the previous vectors according to their multiplicities). Thus, the empirical spectral distributions of $AA^T, \tilde A$ are related through
     \begin{equation} \label{ESDRelation}
        \mu_{\tilde A}\circ T^{-1}=\frac{2M}{M+N} \mu_{AA^T}+\frac{N-M}{M+N} \delta_0 
     \end{equation}
     with $T:\R\to[0, \infty)$ being the map $x\mapsto x^2$.

\bigskip

\textsc{Step function profile}: If $(A_N)_{N\in\N^+}$ is as in Definition \ref{NSStepVP}  with $M(N):=\lceil c N\rceil$ for some $c\in(0, 1]$, then the sequence $(\tilde A_N)_{N\in \N^+}$ is of the form given in Definition \ref{defn rmt step function} with the following modification. We require that there is some $\gamma:\N^+\to\N^+$ with $\lim_{N\to\infty}\gamma(N)=\infty$ so that the $N$-th matrix is of dimension $\gamma(N)\times\gamma(N)$ and, for each $N\in\N^+$, the family $(I_p^{(N)})_{p\in[m]}$ is a partition of $[\gamma(N)]$. The numbers $a_p$ satisfy $\lim_{N\to\infty} R^{(N)}_p/\gamma(N)=a_p.$ With this modification, Theorem \ref{theoremstepfunction} holds if the denominator in \eqref{StepProfLimit} is replaced by $\sqrt{\gamma(N)}$.

The sequence  $(\tilde A_N)_{N\in \N^+}$ fits into this framework. We have $\gamma(N)=\lceil c N\rceil+N$, the role of $m$ (of Definition \ref{defn rmt step function}) is played by $m+n$ ($m, n$ from Definition \ref{NSStepVP}), the $(m+n)^2$ constants are 
\begin{equation}
\tilde \sigma_{p, q}:=\begin{cases}0 & \text{ if } p\in[m], q\in[m],\\
\sigma_{p, q-m} & \text{ if } p\in[m], q\in[m+n]\sm[m],\\
\sigma_{q, p-m} & \text{ if } p\in[m+n]\sm[m], q\in[m],\\
0 & \text{ if } p\in[m+n]\sm[m], q\in[m+n]\sm[m]
\end{cases}
\end{equation}
for each $N$, and the partition of $[\gamma(N)]$ into $m+n$ intervals   consists of the intervals (we write $M$ instead of $\lceil c N\rceil$)
\begin{align}
&\Big\{[Ma_{p-1}, Ma_p)\cap \N^+: p\in[m]\Big\}, \\
&\Big\{[M+N\beta_{q-1}, M+N\beta_q)\cap \N^+: q\in[n]\Big\}. 
\end{align}
Dividing the right endpoints of the intervals by $\gamma(N)$ and taking $N\to\infty$, we get the $m+n$ numbers
\begin{equation}\frac{c}{1+c}a_1<\frac{c}{1+c}a_2<\cdots<\frac{c}{1+c}a_m<\frac{c}{1+c}+\frac{1}{1+c}\beta_1<\cdots<\frac{c}{1+c}+\frac{1}{1+c}\beta_n.
\end{equation}
If we feed these data to the recipe of Definition \ref{defn rmt step function}, relation \eqref{StepModel} will give as $A_N$ the matrix $\tilde A_N$ where $A_N$ is given by \eqref{NSStepModel}. 
The discussion preceding Theorem \ref{theoremstepfunction} applied to the sequence $(\tilde A_N)_{N\ge1}$ gives that the ESD of $\tilde A_N/\sqrt{\gamma(N)}$ converges almost surely weakly to a symmetric probability measure $\tilde \mu^{\sigma}$ with compact support. Call $\tilde \mu^{\sigma}_\infty$ the largest element of the support. Relation \eqref{ESDRelation} implies that the ESD of $A_NA_N^T/N$ converges to a measure with compact support contained in $[0, \infty)$ and the largest element of this support is $\mu_{\infty}=(1+c)(\tilde \mu^{\sigma}_\infty)^2$. Then Theorem \ref{theoremstepfunction} has the following corollary.

\begin{cor}\label{thmgrammatrix}
    Assume that $(A_N)_{N\ge1}$ is as in Definition \ref{NSStepVP}  with $M:=\lceil c N\rceil$ for some $c\in(0, 1]$ and $\E(|X_0|^{4+\delta})<\infty$ for some $\delta>0$. Then it is true that
        \begin{equation}\label{convofgram matr}
           \lim_{N\to\infty} \frac{|A_N A^T_N|_{\op}}{N}=\mu_{\infty} \ \ \ \text{ 
   a.s.}
   \end{equation}
    \end{cor}

\textsc{Continuous function profile}: If $(A_N)_{N\in\N^+}$ is as in Definition \ref{NSContVP}  with $M(N):=\lceil c N\rceil$ for some $c\in(0, 1]$, then we apply the discussion preceding Theorem \ref{theoremstepfunction} to the sequence $(\tilde A_N)_{N\in \N^+}$. The graphon, $W_N$, corresponding to $\tilde A_N$ converges pointwise in $[0, 1]^2$ to the graphon $\tilde \sigma$ with
\begin{equation}
\tilde \sigma(x, y):=\begin{cases} 0 & \text{ if } (x, y)\in[0, c/(1+c)]^2 \cup (c/(1+c), 1]^2,\\
\sigma(x(1+c)/c, (1+c)y-c) & \text{ if } (x, y)\in[0, c/(1+c)]\times (c/(1+c), 1],\\
\sigma(y(1+c)/c, (1+c)x-c) & \text{ if } (x, y)\in(c/(1+c), 1]\times [0, x/(1+c)].
\end{cases} 
\end{equation}
We used \eqref{ContinuousVPApprox} and the continuity of $\sigma$. Since $(\tilde A_N)_{N\in\N^+}$ also satisfies Assumption \ref{BasicAssum}, we get that the ESD of $\tilde A_N/\sqrt{\gamma(N)}$ converges almost surely weakly to a symmetric probability measure $\tilde \mu^{\sigma}$ with compact support. Call $\tilde \mu_{\infty}^\sigma$ the largest element of the support. As above, the ESD of $A_NA_N^T/N$ converges to a measure with compact support contained in $[0, \infty)$, and the largest element of this support is $\mu_{\infty}=(1+c)(\tilde \mu^{\sigma}_\infty)^2$.

\begin{cor}\label{thmgrammatrixcon}
    Assume that $(A_N)_{N\ge1}$ is as in Definition \ref{NSContVP}  with $M:=\lceil c N\rceil$ for some $c\in(0, 1]$ and $\E(|X_0|^{4+\delta})<\infty$ for some $\delta>0$. Then it is true that
        \begin{equation}\label{convofgram matrcont}
           \lim_{N\to\infty} \frac{|A_N A^T_N|_{\op}}{N}=\mu_{\infty} \ \ \ \text{ 
   a.s.}
   \end{equation}
    \end{cor}

\begin{proof} The proof does not follow directly from Corollary \ref{theoremconfunction} because the sequence $(\tilde A_N)_{N\in\N^+}$ does not necessarily have a continuous variance profile in the sense of Definition \ref{definition of rm cont}. Instead, we mimic the proof of that corollary. We define $\sigma^{(n)}$ as in \eqref{Discretization}, and the $M\times N$ matrix $\Sigma^{(n)}_N$ as $ (\Sigma_N^{(n)})_{i, j}:=\sigma^{(n)}(i/M, j/N)$ for all $i\in[M], j\in[N]$. Then we apply an obvious modification of Proposition \ref{symperasma gia proseggiseis pinakon}  (the $N$-th matrix is of dimension $\gamma(N)\times\gamma(N)$, with $\gamma(N)=\lceil cN\rceil+N$) with the role of $\Sigma_N$ and $\Sigma^{(n)}_N$ played by $\tilde \Sigma_{M, N}, \tilde \Sigma_N^{(n)}$ (the symmetrizations of $\Sigma_{M, N}$ and $\Sigma^{(n)}_N$, defined in \eqref{AN_symmetrization}. The proof continues by adopting the proof of Corollary \ref{theoremconfunction} to this setting. Note that $|\tilde \Sigma_{M, N}-\tilde \Sigma_N^{(n)}|_{\max}=|\Sigma_{M, N}-\Sigma_N^{(n)}|_{\max}$, which has $\lim_{n\to\infty}\limsup_{N\to\infty}|\Sigma_{M, N}-\Sigma_{M, N}^{(n)}|_{\max}=0$.
\end{proof}

\begin{rem}
       In \cite{hachem2008clt} the authors showed that if the variances of the entries of $A_{N,M}$ are given by the values of a continuous function (and some extra assumptions such as bounded $4+\epsilon$ moments of the entries) the limiting distribution of the E.S.D. of $A_{N}A^{T}_{N}$ does exist. So in Theorem \ref{thmgrammatrix} we prove the convergence of the largest eigenvalue of these models as well. The authors in \cite{hachem2008clt} also studied the non-centered version of these models, i.e. when the entries of the matrix do not have $0$ mean, but we do not cover this case with our result. 
       \end{rem}

               \subsection{Further applications of Theorem \ref{Theorima gia sxedon kalous pinakes}}\label{applications}
               In the Random Matrix Theory literature what are commonly described as  Random matrices with variance-profile  given by a step function are more or less what we describe in Theorem \ref{theoremstepfunction}. In this subsection we give some examples  which are covered by the generalized version of this variance-profile matrices (Definition \ref{defngenvar}) but not from the "standard" step functions.

               Let $\{a_{i, j}^{(N)}: N\in \N^+, i, j\in [N]\}$ identically distributed random variables, $a_{i, j}^{(N)}=a_{j, i}^{(N)}$ for all $N\in \N^+, i, j\in [N]$, $\{a_{i, j}^{(N)}: 1\le j\le i\le N\}$ independent for each $N$, and $a_{1, 1}^{(1)}$ has mean 0 and variance 1. Fix $p\in(0, 1]$ and let $A_N$ be the matrix with entries
                   \begin{equation} \label{BandDefinition}
                       \{A_{N}\}_{i,j}=a^{(N)}_{i,j} 1_{|i-j|\leq pN}, \ \ \ i,j\in [N] ,
                \end{equation}
             The sequence $(A_N)_{N\in\N^+}$ satisfies Assumption \ref{BasicAssum} (easy to check) and also Assumption \ref{ConvergenceOnTrees}. To see the last point, we follow Remark \ref{RemConvOnTrees}. The graphon corresponding to $A_N$ is $W_N(x, y)=\textbf{1}_{|\lceil Nx\rceil-\lceil Ny\rceil|\le pN}$ which converges to the graphon $W(x, y)=\textbf{1}_{|x-y|\le p}$ at least on the set $\{(x, y)\in[0, 1]^2:|x-y|\ne pn\}$, which has measure 1. Thus, with probability one, the ESD of $A_N/\sqrt{N}$ converges weakly to a symmetric measure $\mu$ with compact support. Call $\mu_\infty$ the supremum of the support of $\mu$.

               \begin{cor}[Non-Periodic Band Matrices with Bandwidth proportional to the dimension] Assume that for the matrix defined in \eqref{BandDefinition} we have that $a_{1, 1}^{(N)}$ has $0$ mean, unit variance and finite $4+\delta$ moment for some $\delta>0$. Then
                   \[\lim_{N \to 
 \infty}\frac{|A_{N}|_{\op}}{\sqrt{N}}=\mu_{\infty} \ \ \ \text{a.s.}\]
               \end{cor}
            \begin{proof}
             The sequence $(A_N)_{N\in\N^+}$ satisfies Assumption \ref{ConvergenceOnTrees}, as we saw above, and also Assumption \ref{Genikes ypotheseis} because $\{a_{i, j}^{(N)}: N\in \N^+, i, j\in [N]\}$ are identically distributed and $a_{1, 1}$ has mean zero, variance one, and finite fourth moment. 
                The corollary then is a straightforward application of Theorem \ref{Theorima gia sxedon kalous pinakes}, where $d_N=3$, the partition of $[N]^2$ required by Definition \ref{defngenvar} consists of the sets
            \begin{align}
                \mathcal{B}^{(N)}_{1}&:=\{(i, j)\in [N]^{2}: |(i/N)-(j/N)|\leq p\}, \\ \mathcal{B}_2^{(N)}&:=\{(i, j)\in [N]^{2}: (i/N)>(j/N)+p \}, \\
                \mathcal{B}_3^{(N)}&:=\{(i, j)\in [N]^{2}: (j/N)>(i/N)+p \},
            \end{align}
                and $s_1^{(N)}=1, s_2^{(N)}= s_3^{(N)}=0$. Condition (b) of that definition is satisfied by $f:=m$ for each $m\in[3]$.
            \end{proof}
            \begin{rem}
                The random band matrix models have been extensively studied after the novel work in \cite{bogachev1991level} and have tremendous application in various research areas. When the bandwidth of the matrices is periodic, i.e., the distance from the diagonal outside which the entries are 0 is periodic, the operator norm has been extensively studied, see for example \cite{sodin2010spectral} or the survey \cite{bourgade2018random}. Moreover when the bandwidth of such matrices is non-periodic but the bandwidth (the maximum number of non identically zero entries per row) is $o(N)$ but also tends to infinity has also been examined in \cite{benaych2014largest}. To the best of our knowledge, the convergence of the largest eigenvalue of non-periodic Band Matrices with bandwidth proportional to the dimension has not been established.
            \end{rem}

Our next result concerns the singular values of triangular random matrices. It is well known under various assumptions for the entries, but we record it here as another application of our main theorem.        

Let $\{a_{i, j}^{(N)}: N\in \N^+, 1\le i\le j\in [N]\}$ identically distributed random variables, $\{a_{i, j}^{(N)}: 1\le j\le i\le N\}$ independent for each $N$, and $a_{1, 1}^{(1)}$ has mean 0 and variance 1. Let $A_N$ be the matrix with entries
                   \begin{equation} \label{TriangDefinition}
                       \{A_{N}\}_{i,j}=a^{(N)}_{i,j} 1_{i\le j}, \ \ \ i,j\in [N] ,
                \end{equation}

                   \begin{cor}[Triangular matrices]
Assume that for the matrix defined in \eqref{TriangDefinition} we have that $a_{1, 1}^{(N)}$ has $0$ mean, unit variance and finite $4+\delta$ moment for some $\delta>0$. Then
\[\lim_{N \to \infty}\frac{|A_{N}A^{T}_{N}|_{\op}}{N}= e \ \ \ \text{a.s.}\]
            \end{cor}
 \begin{proof}
     As in the case of Gram matrices, we denote by $\tilde A_N$ the symmetrization of $A_N$, defined in \eqref{AN_symmetrization}. We have $|A_N A_N^T|_{\op}=|\tilde A_N|_{\op}^2$. We will apply Theorem \ref{Theorima gia sxedon kalous pinakes} to the sequence $(\tilde A_N)_{N\in\N^+}$. The partition of $[2N]^2$ required by Definition \ref{defngenvar} consists of the following three sets (i.e., $d_N=3$)
     \begin{align} 
     \mathcal{B}^{(N)}_1&:=\{(i, j)\in [2N]^2: |i-j|\le N-1\}, \\
     \mathcal{B}^{(N)}_2&:=\{(i, j)\in [2N]^2: i\ge N+j\}, \\
     \mathcal{B}^{(N)}_3&:=\{(i, j)\in [2N]^2: j\ge N+i\}, 
     \end{align}
 and the corresponding values of the variance are $s_1^{(N)}=0, s_2^{(N)}=s_3^{(N)}=1$. Assumption \ref{ConvergenceOnTrees} follows as an application of Remark \ref{RemConvOnTrees}, in the same way as in the previous corollary. The measure $\mu$ of that assumption satisfies $\mu\circ T^{-1}=\nu$, where $\nu$ is the limit of the E.S.D of $N^{-1}A_{N}A^{T}$ [recall \eqref{ESDRelation}]. It was shown in \cite{cheliotis2022singular} that $\nu$ has support $[0, e]$. It follows that $\mu$ has support $[-\sqrt{e}, \sqrt{e}]$ [See Remark 2.2 of \cite{bose2022xx} for a more detailed discussion of this phenomenon].

 Assumption \ref{Genikes ypotheseis} is satisfied because the elements of $\tilde A_N$  with indices in $\mathcal{B}^{(N)}_2\cup \mathcal{B}^{(N)}_3$ are identically distributed with zero mean, unit variance and finite fourth moment (the remaining elements of the matrix are identically zero random variables). Finally, condition \eqref{extrassuconvex} is satisfied as equality. 

 Thus, the corollary follows from Theorem \ref{Theorima gia sxedon kalous pinakes}.
 \end{proof}

\section{Two technical lemmas}
\label{CountLemma}

In the next lemma, we prove the crucial estimate we invoked in the proof of Proposition \ref{protasi gia megalo trace}. We adopt and present the terminology of Section 5.1.1 of \cite{bai2010spectral}.

\begin{lem} \label{CycleCountLemma} $N_{T, a_1, a_2, \ldots, a_s}\le (4k^4)^{4(s+1-t)+2(k-s)}$ if $t\ge2$ and 
$N_{T, a_1, a_2, \ldots, a_s}=1$ if $t=1$.
\end{lem}

\begin{proof}
When $t=1$, since the cycle is bad, we have $s=1$ and $a_1=2k$, and there is only one cycle with these $s, t$ and vertex set $\{1\}$.

For the case $t\ge2$, take a cycle $\B{i}:=(i_1, i_2, \ldots, i_{2k})$ as in \eqref{SpecBadCycles} and assume that it has edge multiplicities $a_1, a_2, \ldots, a_s\ge 2$.
Each step in the cycle we call a \textit{leg}. More formally, legs are the elements of the set $\{(r, (i_r, i_{r+1})): r=1, 2,\ldots, 2k\}$. Edges of the cycle we call the edges of $G(\B{i})$, and the multiplicity of each edge is computed from $\B{i}$. The graph $G(\B{i})$ does not have multiple edges.

For $1\le a<b$, we say that the leg $(a, (i_a, i_{a+1}))$ is \textit{single} up to $b$ if $\{i_a, i_{a+1}\}\ne \{i_c, i_{c+1}\}$ for every $c\in\{1, 2, \ldots, b-1\}, c\ne a$. We classify the $2k$ legs of the cycle into 4 sets $T_1, T_2, T_3, T_4$. The leg $(a, (i_a, i_{a+1}))$  belongs to 

\smallskip

$T_1$: if $i_{a+1}\notin \{i_1, \ldots, i_a\}$. I. e., the leg leads to a new vertex.

\smallskip

$T_3$: if there is a $T_1$ leg $(b, (i_b, i_{b+1}))$ with $b<a$ so that $a=\min\{c>b: \{i_c, i_{c+1}\}=\{i_b, i_{b+1}\}\}$. I. e., at the time of its appearance, it increases the multiplicity of a $T_1$ edge of $G(\B{i})$ from 1 to 2.

\smallskip

$T_4$: if it is not $T_1$ or $T_3$.

\smallskip

$T_2$: if it is $T_4$ and there is no $b<a$ with $\{i_a, i_{a+1}\}=\{i_b, i_{b+1}\}$. \\
\phantom{That i} I.e., at the time of its appearance, it creates a new edge but leads to a vertex that has \\ \phantom{That i} appeared already.

\smallskip

\noindent Moreover, a $T_3$ leg $(a, (i_a, i_{a+1}))$ is called \textit{irregular} if there is exactly one $T_1$ leg $(b, (i_b, i_{b+1}))$ which has $b<a$, $v_a\in\{i_b, i_{b+1}\}$, and is single up to $a$. Otherwise the leg is called \textit{regular}. \\
It is immediate that a $T_4$ leg is one of the following three kinds.

\smallskip

a) It is a $T_2$ leg.

\smallskip

b) Its appearance increases the multiplicity of a $T_2$ edge from 1 to 2.

\smallskip

c) Its edge marks the third or higher order appearance of an edge. 

\smallskip

\noindent The number of edges of  $G(\B{i})$ is $s$ and the number of its vertices is $t$ (since $T(\B{i})\sim T\in \B{C}_{t-1}$). Call 

\smallskip

$\ell$: the number of edges of  $G(\B{i})$ that have multiplicity at least 3. 

\smallskip

$m$: the number of $T_2$ legs.

\smallskip

$r$: the number of regular $T_3$ legs.

\smallskip

\noindent We have for $r, t,$ and $|T_4|$ the following bounds
\begin{align} r&\le 2m, \label{rBound}\\
t&=s+1-m\le k, \label{tBound}\\
|T_4| &=2m+2(k-s). \label{T4Bound}
\end{align}
The first relation is Lemma 5.6 in \cite{bai2010spectral}. The second is true because if we remove the $m$ edges traveled by $T_2$ legs, we get a tree with $s-m$ edges and $t$ vertices, and in any tree the number or vertices equals the number of edges plus one. Then the inequality is true because $s\le k$ (all edges of  $G(\B{i})$  have multiplicity at least 2) and if $s=k$, then $m\ge1$ since the cycle is bad. For the last relation, note that $|T_3|=|T_1|=t-1$ and thus, using \eqref{tBound} too, we have
$|T_4|=2k-2(t-1)=2k-2(s-m).$ \\
Now back to the task of bounding $N_{T, a_1, \ldots, a_s}$.
We fix a cycle as in the beginning of the proof and we record
\begin{itemize}
\item for each $T_4$ leg, a) its order in  the cycle, b) the index of its initial vertex, c) the index of its final vertex, and d) the index of the final vertex of the next leg in case that leg is $T_1$. This gives a $Q_1\subset \{1, 2, \ldots, 2k\}\times (\{1, 2, \ldots t\}^2\cup \{1, 2, \ldots t\}^3)$ with $|T_4|$ elements.

\item for each regular $T_3$ leg, a) its order in the cycle, b) the index of its initial vertex, and c) the index of its final vertex. This gives a $Q_2\subset \{1, 2, \ldots, 2k\}\times \{1, 2, \ldots, t\}^2$ with $r$ elements.
\end{itemize}
We call $U$ the set of all indices that appear as fourth coordinate in elements of $Q_1$.  These are indices of final vertices of $T_1$ legs.\\
We claim that, having $Q_1, Q_2$ and knowing that $T(\B{i})=T$, we can reconstruct the cycle \B{i}.

We determine what kind each leg of the cycle is and what the index of its initial and its final vertex is. These data are known for the $T_4$ and $T_3$ regular legs. The remaining legs are $T_1$ or $T_3$ irregular. We discover the nature of each of them by traversing the cycle from the beginning as follows. The first leg is $T_4$ (if $i_2=i_1$) or $T_1$. The set $Q_1$ will tell us if we are in the first case and will give us all we want. If we are in the second case, the initial vertex is 1 and the final 2. Assume that we have arrived at a vertex $v_i$ in the cycle with the smallest $i$ for which the nature of the leg $\ell_i:=(i, (v_i, v_{i+1}))$ is not known yet. If the vertex $v_i$ has no neighbors in  $G(\B{i})$ that we haven't encountered up to the leg $\ell_{i-1}$, then $\ell_i$ is $T_3$ irregular, and by the defining property of $T_3$ irregular legs, we can determine the index of its final vertex. If the vertex $v_i$ does have such neighbors, call $z$ the one that appears earlier in the cycle. \\
\begin{figure}[ht] 
  \centering
  \includegraphics[width=14em]{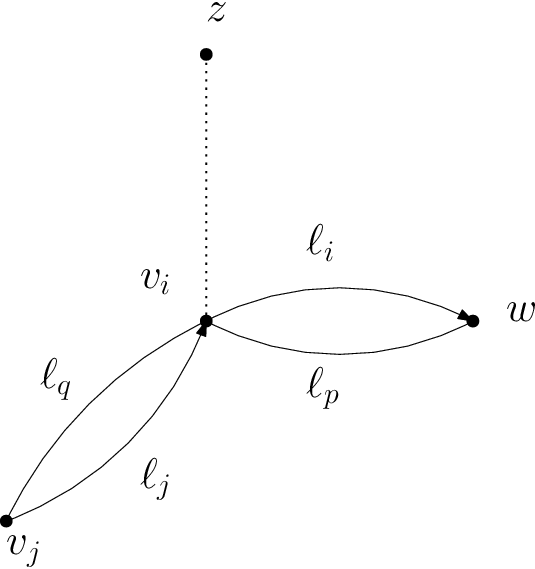} 
  \caption{The case $z\notin U$.The legs $\ell_i, \ell_j (i<j)$ are $T_3$, while $\ell_p, \ell_q$ are $T_1$.}\label{graphCase}
\end{figure}
$\bullet$ If $z\in U$, then in case it was included in $U$ because of $\ell_{i-1}$ (this can be read off from $Q_1$. Note that $z$ could not have been included because of an earlier leg because $z$ has not appeared earlier than $v_i$), we have that $\ell_i$ is $T_1$ with $v_{i+1}=z$, while in case it was included with a leg $\ell_{i'}$ with index $i'\ge i$, we have that $\ell_i$ can't be $T_1$ (because then $v_{i+1}$ would be a neighbor of $v_i$ appearing earlier than $z$, contradicting the choice of $z$), thus $\ell_i$ is $T_3$ irregular.

$\bullet$ If $z\notin U$, we will show that $\ell_i=(i, (v_i, w))$ is $T_1$. Assume on the contrary that it is $T_3$ irregular. Clearly $z\ne w$, and call $\ell_p$ ($p<i$) the $T_1$ leg that has vertices $v_i, w$ and is single up to $i-1$. The cycle will visit the vertex $v_i$ at a later point, with a leg $\ell_j=(j, (v_j, v_i))$ with $j>i$ and $v_j\ne z, v_j\ne v_i$
, in order to create the edge that connects $v_i$ with $z$ (that is, $\ell_{j+1}=(j+1, (v_i, z))$ will be $T_1$), see Figure \ref{graphCase}. The leg $\ell_j$ is not $T_1$ because $v_i$ has been visited by an earlier leg, and it is not $T_4$ because we assumed that $z\notin U$. 
It has then to be $T_3$. Thus, there is a leg $\ell_q$ connecting vertices $v_i, v_j$ that is $T_1$. 

If $q<i$, then we consider two cases. If $v_j=w$, then $\ell_j$ is $T_4$, because the edge $v_i, w$ has been traveled already by $\ell_p, \ell_i$ (recall that $p<i<j$), and this would force $z\in U$, a contradiction. If $v_j\ne w$, then $\ell_i$ would have been $T_3$ regular as there are at least two $T_1$ legs (i.e., $\ell_p, \ell_q$) with order less than $i$ with one vertex $v_i$, traveling different edges, and single up to $i-1$, again a contradiction because $\ell_i$ is $T_1$ or $T_3$ irregular.

 If $q>i$, then $v_j (\ne z)$ is a neighbor of $v_i$ (that is, the $T_1$ leg $\ell_q$ goes from $v_i$ to $v_j$) that appears after leg $\ell_i$ but earlier than $z$, which contradicts the definition of $z$. We conclude that $\ell_i$ is $T_1$.

Thus, having $T, Q_1, Q_2$ allows to determine $\B{i}$.

\noindent The above imply that the number of bad cycles with given $T, t, r$ is at most
\begin{equation}
 (2kt^2(t+1))^{|T_4|} (2k t^2)^r \le (4k^4)^{r+|T_4|}. \label{MapCount}
\end{equation}
Then \eqref{rBound} and \eqref{T4Bound} give 
$r+|T_4|\le 4m+2(k-s)$, and finally using \eqref{tBound}, we get the desired bound. 
\end{proof}

The next lemma is used in the proof of Theorem \ref{Theorima gia sxedon kalous pinakes}.

\begin{lem} \label{BreakLemma}
    Let $(A_N)_{N\in\N^+}$ be a sequence of matrices, $A_N$ of dimension $N\times N$,  that satisfies Assumption \ref{BasicAssum} and Assumption \ref{ConvergenceOnTrees} with measure $\mu$. Suppose that there are two sequences of matrices $(A^{(1)}_N)_{N\in\N^+}$ and $(A^{(2)}_N)_{N\in \N^+}$ such that
    \begin{itemize}
        \item[(a)] $A_{N}=A^{(1)}_N+A_N^{(2)}$, 
        \item[(b)] For all $N\in\N^+$ and $i, j\in[N]$, at least one of $(A^{(1)}_N)_{i, j}, (A^{(2)}_N)_{i, j}$ is identically zero random variable. 
        \item[(c)] $\mu_{A_N^{(1)}/\sqrt{N}}\Rightarrow \mu \text{ in probability as } N\to\infty.$
    \end{itemize}
    Then $(A_N^{(1)})_{N\in\N^+}$ also satisfies Assumptions \ref{BasicAssum}, \ref{ConvergenceOnTrees} with the measure $\mu$.
\end{lem}
\begin{proof} We only need to check the validity of Assumption \ref{ConvergenceOnTrees} as the validity of Assumption \ref{BasicAssum} is immediate. \\
Because $A_N^{(1)}$ satisfies Assumption \ref{BasicAssum}, there is a decreasing sequence $(\eta_N)_{N\in\N^+}$ of positive reals converging to 0 so that
\begin{equation} \label{TruncProperty}
 \lim_{N\to\infty} \frac{1}{N^2}\sum_{i, j\in [N]}   \E \left[ (\{A_N^{(1)}\}_{i,j})^2 \mathbf{1}\left( |\{A_{N}^{(1)}\}_{i,j}|>\eta_N N^{\frac{1}{2}}  \right)\right]=0.  
\end{equation}
Set $A_N^{(1), \leq }$ to be the matrix whose $(i, j)$ entry is 
    \begin{align}
    \{A_N^{(1)}\}_{i,j} \mathbf{1}\left( |\{A_{N}^{(1)}\}_{i,j}| \leq  \eta_N N^{\frac{1}{2}} \right)- \E \left[ \{A_N^{(1)}\}_{i,j} \mathbf{1}\left( |\{A_{N}^{(1)}\}_{i,j}|  \leq  \eta_N N^{\frac{1}{2}}  \right)\right]
\end{align}
and $\mu_{N, i, j}:=\E \left[ \{A_N^{(1)}\}_{i,j} \mathbf{1}\left( |\{A_{N}^{(1)}\}_{i,j}|  \leq  \eta_N N^{\frac{1}{2}}  \right)\right]$. 

\textsc{Claim}:
\begin{equation}
   \mu_{N^{-1/2} A_{N}^{(1), \le}} \Rightarrow \mu \quad \text{ in probability as } N\to\infty \label{TrCentConvergence}
\end{equation}

The Levy distance between $\mu_{A_N^{(1)}/\sqrt{N}}$ and $\mu_{A_N^{(1), \le}/\sqrt{N}}$ is bounded as follows.
\begin{align}
L^3(\mu_{A_N^{(1)}/\sqrt{N}}, \mu_{A_N^{(1), \le}/\sqrt{N}}) & \le \frac{1}{N}\text{tr}\bigg\{\bigg(\frac{1}{\sqrt{N}} A_N^{(1)}-\frac{1}{\sqrt{N}} A_N^{(1), \le}\bigg)^2\bigg\}\\&=\frac{1}{N^2}\sum_{i, j\in [N]} \{\mu_{N, i, j}^2 \mathbf{1}\left( |\{A_{N}^{(1)}\}_{i,j}| \leq  \eta_N N^{\frac{1}{2}} \right)+(\{A_N^{(1)}\}_{i, j}+\mu_{N, i, j})^2 \mathbf{1}\left( |\{A_{N}^{(1)}\}_{i,j}|>\eta_N N^{\frac{1}{2}} \right)\} \\& \le \frac{3}{N^2}\sum_{i, j\in [N]} \mu_{N, i, j}^2 +\frac{2}{N^2} \sum_{i, j\in[N]} (\{A_N^{(1)}\}_{i, j})^2 \mathbf{1}\left( |\{A_{N}^{(1)}\}_{i,j}|>\eta_N N^{\frac{1}{2}} \right). \label{LevyDistBound}
\end{align}

Since the entries of $A_N^{(1)}$ have mean 0, we have
$$\mu_{N, i, j}^2=\left(\E \left[ \{A_N^{(1)}\}_{i,j} \mathbf{1}\left( |\{A_{N}^{(1)}\}_{i,j}|>\eta_N N^{\frac{1}{2}}  \right)\right]\right)^2\le \E\left[  (\{A_N^{(1)}\}_{i, j})^2 \mathbf{1}\left( |\{A_{N}^{(1)}\}_{i,j}|>\eta_N N^{\frac{1}{2}} \right) \right].$$
Thus, the expectation of the expression in \eqref{LevyDistBound} is at most
$$\frac{5}{N^2} \sum_{i, j\in[N]} \E \bigg\{(\{A_N^{(1)}\}_{i, j})^2 \mathbf{1}\left( |\{A_{N}^{(1)}\}_{i,j}|>\eta_N N^{\frac{1}{2}} \right) \bigg\},$$
which tends to zero as $N\to\infty$ due to \eqref{TruncProperty}. This, combined with assumption (c), proves the claim.

Fix $k\in\N^+$ and set $M_N(k), M^{(1), \leq}_{N}(k)$ the asymptotic contributing terms (see \eqref{orismos kalon orwn}) of $A_N$ and $A^{(1),\leq}_N$ respectively. Notice that
\begin{equation} \label{MIneq}
M^{(1), \leq}_N(k)\leq M^{(1)}_{N}(k)\leq M_N(k).
\end{equation}
The rightmost inequality is true because the variance of $\{A_N^{(1)}\}_{i, j}$ is either zero or $s_{i, j}^{(N)}$ due to assumption (b) of the lemma. The leftmost inequality is true because if $W$ is a real valued random variable with mean 0 and finite variance and $\tilde W$ is a variable with $|\tilde W|\le |W|$, then $\text{Var}(\tilde W)\le \text{Var}(W)$.

Lemma 3.6 of \cite{zhu2020graphon} implies that 
\begin{align}\label{ineqformomentstruncated}
    \frac{M^{(1), \leq}_{N}(k)}{N^{k+1}}=\frac{1}{N^{k+1}} \E \tr \{(A^{(1), \leq}_N)^{2k}\}+o(1) 
\end{align}
as $N\to\infty$. We will prove that the right hand side converges to $\int  x^{2k} d\mu$ as $N\to\infty$. It will be convenient to let $B_N:=A_N^{(1), \le }/\sqrt{N}$ and $\{\lambda_i(B_N): i\in [N]\}$ its eigenvalues.

Pick some $C>\mu_\infty$ and consider the function $g_C(x)=(|x|\wedge C)^{2k}$, which is bounded and continuous. Then, 
\begin{equation} \label{traceAndEv}
\frac{1}{N^{k+1}} \E \tr \{(A^{(1), \leq}_N)^{2k}\}=\frac{1}{N}\sum_{i=1}^{N}\E \{(\lambda_i\{B_N\})^{2k}\}
\end{equation}
and the right hand side can be estimated as follows.
    \begin{equation}\label{divergBound} 
    \begin{aligned}
    &\left|\frac{1}{N}\sum_{i=1}^{N}\E \{(\lambda_i\{B_N\})^{2k}\}-\frac{1}{N}\sum_{i=1}^{N} \E g_C(\lambda_{i}\{B_N\})\right| \leq  \frac{1}{N}\sum_{i=1}^N \E\{\lambda^{2k}_{i}(B_N)\mathbf{1}_{|\lambda_{i}(B_N)|\geq C}\}\\
       & \leq \frac{1}{N}\sum_{i=1}^{N} \sqrt{\E \lambda^{4k}_i(B_N)}\sqrt{\mathbb{P}(|\lambda_i(B_N)|\geq C)} \leq \sqrt{\sum_{i=1}^{N}\frac{\E \lambda_i^{4k}(B_{N})}{N}} \sqrt{\ \frac {\E\sum_{i=1}^{N}\textbf{1}_{ |\lambda_{i}(B_N)|\geq C}}{N}}\overset{N\to\infty}{\to} 0.
    \end{aligned}\end{equation}
    To justify the convergence to zero, note that the quantity in the second square root converges to zero by our choice of $C>\mu_{\infty}$ and the in probability weak convergence of the E.S.D. of $B_N$ to $\mu$. The quantity in the first square root is bounded in $N$ because, due to \eqref{ineqformomentstruncated}, its difference from $M_N^{(1), \le}(2k)/N^{2k+1}$ is bounded and the latter is less than $M_N(2k)/N^{2k+1}$ which is bounded in $N$ since it converges to $\int x^{4k}\, d\mu$.

    The in probability weak convergence \eqref{TrCentConvergence} implies that 
    \begin{align}
        \frac{1}{N}\sum_{i=1}^{N} g_C(\lambda_{i}\{B_N\}) \to \int x^{2k}d\mu \text{ in probability, }
    \end{align}
    and the boundedness of $g_C$ allows to conclude that 
    \begin{align} \label{WCApplication}
          \lim_{N\to\infty} \frac{1}{N}\sum_{i=1}^{N} \E g_C(\lambda_i\{B_N\})=\int x^{2k}d\mu. 
    \end{align}
     Thus, relations \eqref{ineqformomentstruncated}, \eqref{traceAndEv}, \eqref{divergBound},\eqref{WCApplication} show that 
    \begin{equation}
         \lim_{N\to\infty} \frac{M^{(1), \leq}_{N}(k)}{N^{k+1}}=\int x^{2k}d\mu. 
    \end{equation}
And this combined with \eqref{MIneq} concludes the proof.
\end{proof}

\nocite{*} 
\bibliographystyle{abbrv}

\bibliography{references}

\end{document}